\documentclass[acmtog]{acmart}
\usepackage{mathtools}
\usepackage{amsmath}
\usepackage{amsthm}
\usepackage{setspace}
\usepackage{threeparttable}
\usepackage{booktabs}
\usepackage{stfloats}
\usepackage{textcomp}
\usepackage[inline]{enumitem}
\usepackage{algorithm}
\usepackage{algorithmicx}
\usepackage{algpseudocode}
\usepackage{caption}
\usepackage{subcaption}
\usepackage{multirow}
\usepackage{bm}
\usepackage{bbm}
\usepackage{physics}    % many convenient math definitions
\usepackage[frozencache]{minted}
\usepackage{nicefrac}

\usepackage[nameinlink,noabbrev]{cleveref}
\renewcommand{\cref}[1]{\Cref{#1}}
\newcommand{\eqnref}[1]{\hyperref[eqn:#1]{(\ref*{eqn:#1})}}

\newcommand{\V}[1]{\bm{#1}}
\renewcommand{\real}{\mathbb{R}}
\newcommand{\nth}[1]{#1^{\text{th}}}
\newcommand{\trans}[1]{#1^\intercal}
\newcommand{\invtrans}[1]{#1^{-\intercal}}
\newcommand{\transv}[1]{\trans{\V{#1}}}
\newcommand{\invtransv}[1]{\invtrans{\V{#1}}}
\newcommand{\indicator}[1]{\mathbbm{1}_{#1}}
\newcommand{\coeffof}[2]{\mathfrak{C}_{a}^{(#1)}\left[#2\right]}
\newcommand{\defeq}{\vcentcolon=}

\DeclareMathOperator*{\argmin}{arg\,min}

\newtheorem{prop}{Proposition}
\newtheorem{definition}{Definition}

\makeatletter
\newcommand*\bigcdot{\mathpalette\bigcdot@{.5}}
\newcommand*\bigcdot@[2]{\mathbin{\vcenter{\hbox{\scalebox{#2}{$\m@th#1\bullet$}}}}}
\makeatother

\newcommand{\vx}{\V{x}}
\newcommand{\vX}{\V{X}}
\newcommand{\vXk}{\V{X_k}}
\newcommand{\vxs}{\V{x^*}}
\newcommand{\vxz}{\V{x_0}}
\newcommand{\vxi}{\V{x_i}}
\newcommand{\vxk}{\V{x_k}}
\newcommand{\vf}{\V{f}}
\newcommand{\vF}{\V{F}}
\newcommand{\vR}{\V{R}}
\newcommand{\vFi}{\V{F_i}}
\newcommand{\vz}{\V{0}}
\newcommand{\vfi}{\V{f_i}}
\newcommand{\vfk}{\V{f_k}}
\renewcommand{\vv}{\V{v}}

\newcommand{\vp}{\V{p}}
\newcommand{\vP}{\V{P}}
\newcommand{\vq}{\V{q}}
\newcommand{\vH}{\V{H}}
\newcommand{\vU}{\V{U}}
\newcommand{\vW}{\V{W}}
\newcommand{\vM}{\V{M}}
\newcommand{\vA}{\V{A}}
\newcommand{\vSig}{\V{\Sigma}}
\newcommand{\vUz}{\V{U_0}}
\newcommand{\vWz}{\V{W_0}}
\newcommand{\vPz}{\V{P_0}}
\newcommand{\vXz}{\V{X_0}}
\newcommand{\vSigz}{\V{\Sigma_0}}
\newcommand{\vUk}{\V{U_k}}
\newcommand{\vWk}{\V{W_k}}
\newcommand{\vPk}{\V{P_k}}
\newcommand{\vSigk}{\V{\Sigma_k}}
\newcommand{\vV}{\V{V}}
\newcommand{\vE}{\V{E}}
\newcommand{\vB}{\V{B}}
\newcommand{\vBu}{\V{B_u}}
\newcommand{\vBw}{\V{B_w}}
\newcommand{\vBp}{\V{B_p}}
\newcommand{\vI}{\V{I}}
\newcommand{\vD}{\V{D}}
\newcommand{\vsig}{\V{\sigma}}
\DeclareMathOperator{\vect}{vec}
\DeclareMathOperator{\diag}{diag}
\DeclareMathOperator{\rms}{RMS}
\newcommand{\matrow}[2]{#1_{\bigcdot #2}}
\newcommand{\matcol}[2]{#1_{#2 \bigcdot}}

\makeatletter
\newcommand{\algmargin}{\the\ALG@thistlm}
\makeatother
\newlength{\whilewidth}
\algnewcommand{\parState}[1]{\State%
  \parbox[t]{\dimexpr\linewidth-\algmargin}{\strut \hangindent2em #1\strut}}

\DeclareSymbolFont{yhlargesymbols}{OMX}{yhex}{m}{n}
\DeclareMathAccent{\wideparen}{\mathord}{yhlargesymbols}{"F3}

\makeatletter\let\expandableinput\@@input\makeatother

\acmJournal{TOG}
\acmYear{2021}\acmVolume{40}\acmNumber{4}\acmArticle{79}\acmMonth{8}
\acmDOI{10.1145/3450626.3459755}

\setcopyright{rightsretained}
\title{SANM: A Symbolic Asymptotic Numerical Solver with Applications in Mesh
Deformation}

\begin{document}
% DO NOT ENTER AUTHOR INFORMATION FOR ANONYMOUS TECHNICAL PAPER SUBMISSIONS TO SIGGRAPH 2019!
\author{Kai Jia}
\orcid{0000-0001-8215-9899}
\affiliation{%
  \institution{MIT CSAIL}
  \streetaddress{32 Vassar St}
  \city{Cambridge}
  \state{MA}
  \postcode{02139 }
  \country{USA}}
\email{jiakai@mit.edu}
%\author{Valerie B\'eranger}
%\affiliation{%
%  \institution{Inria Paris-Rocquencourt}
%  \city{Rocquencourt}
%  \country{France}
%}
%\email{beranger@inria.fr}
%\author{Aparna Patel}
%\affiliation{%
% \institution{Rajiv Gandhi University}
% \streetaddress{Rono-Hills}
% \city{Doimukh}
% \state{Arunachal Pradesh}
% \country{India}}
%\email{aprna_patel@rguhs.ac.in}
%\author{Huifen Chan}
%\affiliation{%
%  \institution{Tsinghua University}
%  \streetaddress{30 Shuangqing Rd}
%  \city{Haidian Qu}
%  \state{Beijing Shi}
%  \country{China}
%}
%\email{chan0345@tsinghua.edu.cn}
%\author{Ting Yan}
%\affiliation{%
%  \institution{Eaton Innovation Center}
%  \city{Prague}
%  \country{Czech Republic}}
%\email{yanting02@gmail.com}
%\author{Tian He}
%\affiliation{%
%  \institution{University of Virginia}
%  \department{School of Engineering}
%  \city{Charlottesville}
%  \state{VA}
%  \postcode{22903}
%  \country{USA}
%}
%\affiliation{%
%  \institution{University of Minnesota}
%  \country{USA}}
%\email{tinghe@uva.edu}
%\author{Chengdu Huang}
%\author{John A. Stankovic}
%\author{Tarek F. Abdelzaher}
%\affiliation{%
%  \institution{University of Virginia}
%  \department{School of Engineering}
%  \city{Charlottesville}
%  \state{VA}
%  \postcode{22903}
%  \country{USA}
%}

%\renewcommand\shortauthors{Zhou, G. et al}

\begin{abstract}
    Solving nonlinear systems is an important problem. Numerical continuation
    methods efficiently solve certain nonlinear systems. The Asymptotic
    Numerical Method (ANM) is a powerful continuation method that usually
    converges faster than Newtonian methods. ANM explores the landscape of the
    function by following a parameterized solution curve approximated with a
    high-order power series. Although ANM has successfully solved a few graphics
    and engineering problems, prior to our work, applying ANM to new problems
    required significant effort because the standard ANM assumes quadratic
    functions, while manually deriving the power series expansion for
    nonquadratic systems is a tedious and challenging task.

    This paper presents a novel solver, SANM, that applies ANM to solve
    symbolically represented nonlinear systems. SANM solves such systems in a
    fully automated manner. SANM also extends ANM to support many nonquadratic
    operators, including intricate ones such as singular value decomposition.
    Furthermore, SANM generalizes ANM to support the implicit homotopy form.
    Moreover, SANM achieves high computing performance via optimized system
    design and implementation.

    We deploy SANM to solve forward and inverse elastic force equilibrium
    problems and controlled mesh deformation problems with a few constitutive
    models. Our results show that SANM converges faster than Newtonian solvers,
    requires little programming effort for new problems, and delivers comparable
    or better performance than a hand-coded, specialized ANM solver. While we
    demonstrate on mesh deformation problems, SANM is generic and potentially
    applicable to many tasks.
\end{abstract}

% vim: tw=80 filetype=tex foldmethod=marker foldmarker=f{{{,f}}} spell spelllang=en

% generated by the tool at http://dl.acm.org/ccs.cfm



\begin{teaserfigure}
    \begin{subfigure}[t]{0.245\textwidth}
        \centering
        \captionsetup{width=.8\linewidth}
        \includegraphics[width=\linewidth]{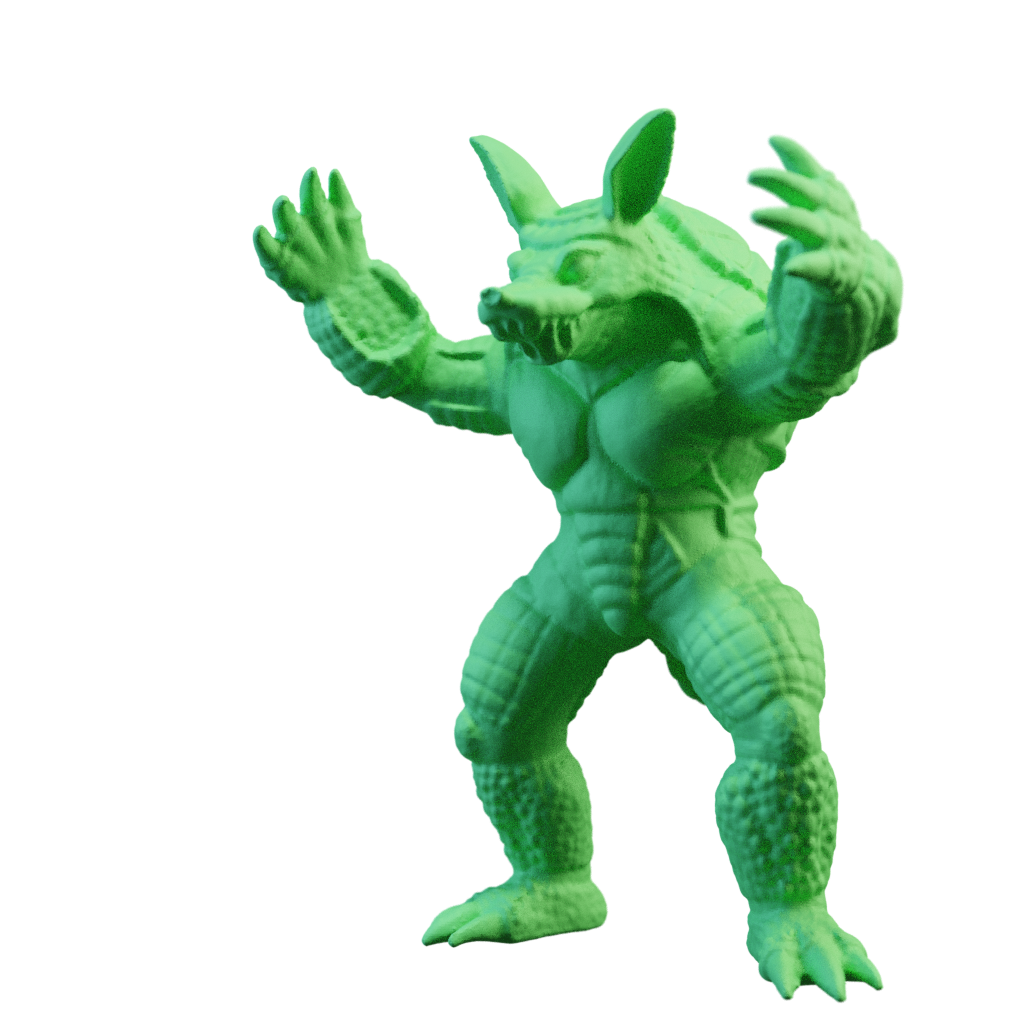}
        \caption{Inverse gravity equilibrium with incompressible neo-Hookean
            material. This figure shows a rest shape that will deform to the
            original Armadillo under gravity.}
        \label{fig:summary:a}
    \end{subfigure}
    \hfill
    \begin{subfigure}[t]{0.245\linewidth}
        \centering
        \captionsetup{width=.8\linewidth}
        \includegraphics[width=\linewidth]{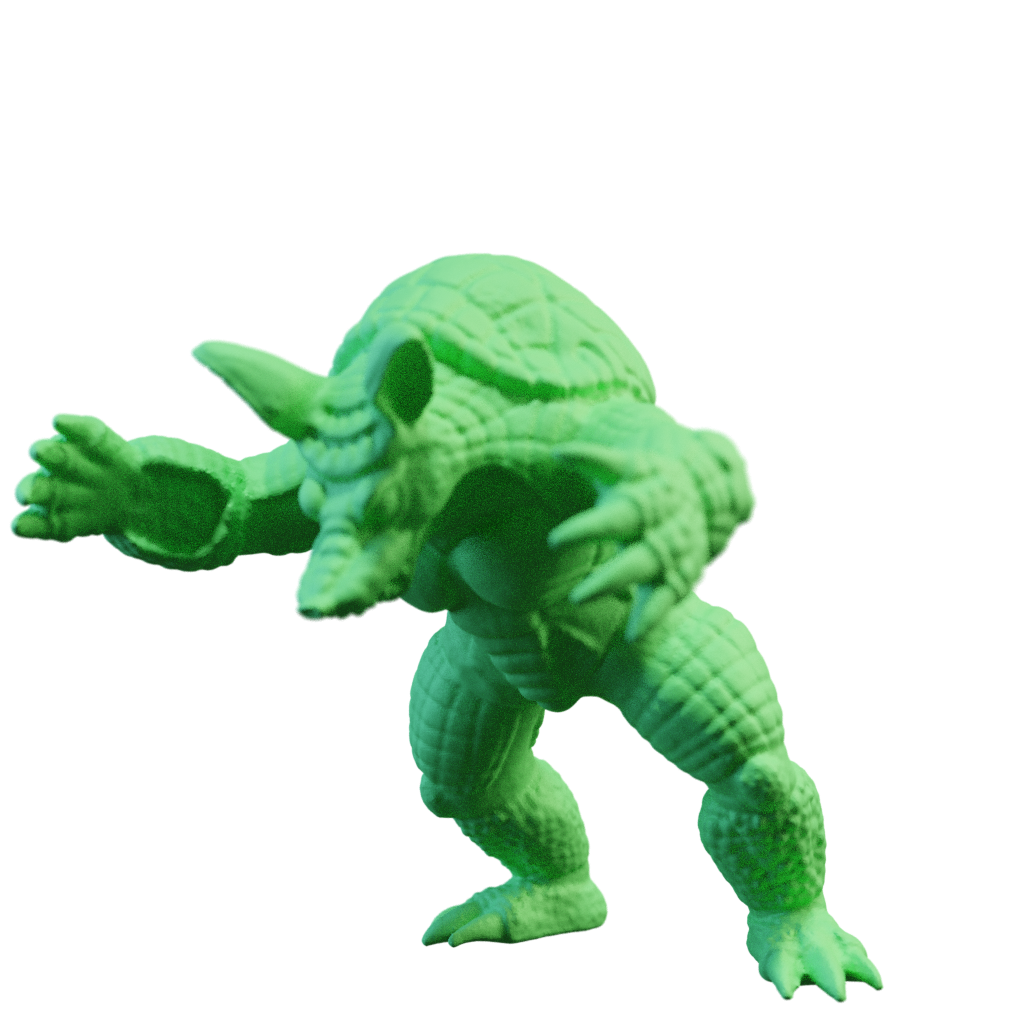}
        \caption{Gravity equilibrium with compressible neo-Hookean material.
            This figure shows the final shape to which an Armadillo will deform
            under gravity.}
        \label{fig:summary:b}
    \end{subfigure}
    \hfill
    \begin{subfigure}[t]{0.245\linewidth}
        \centering
        \captionsetup{width=.8\linewidth}
        \includegraphics[width=\linewidth]{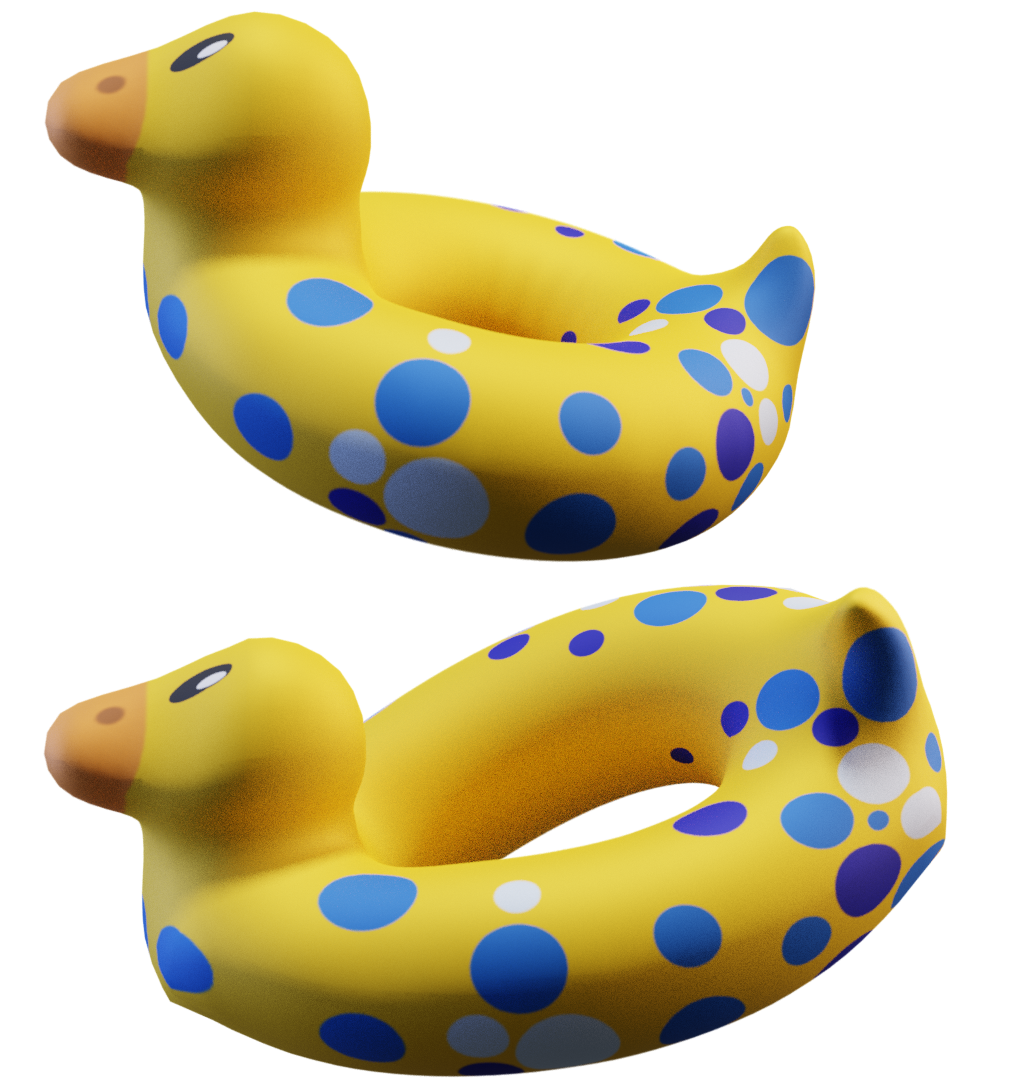}
        \caption{Controlled deformation of the Bob model with incompressible
            neo-Hookean material by fixing the head and moving the tail. Top:
            rest shape. Bottom: deformed shape. Model created by Keenan Crane.
        }
        \label{fig:summary:c}
    \end{subfigure}
    \hfill
    \begin{subfigure}[t]{0.245\linewidth}
        \centering
        \captionsetup{width=.8\linewidth}
        \includegraphics[width=\linewidth]{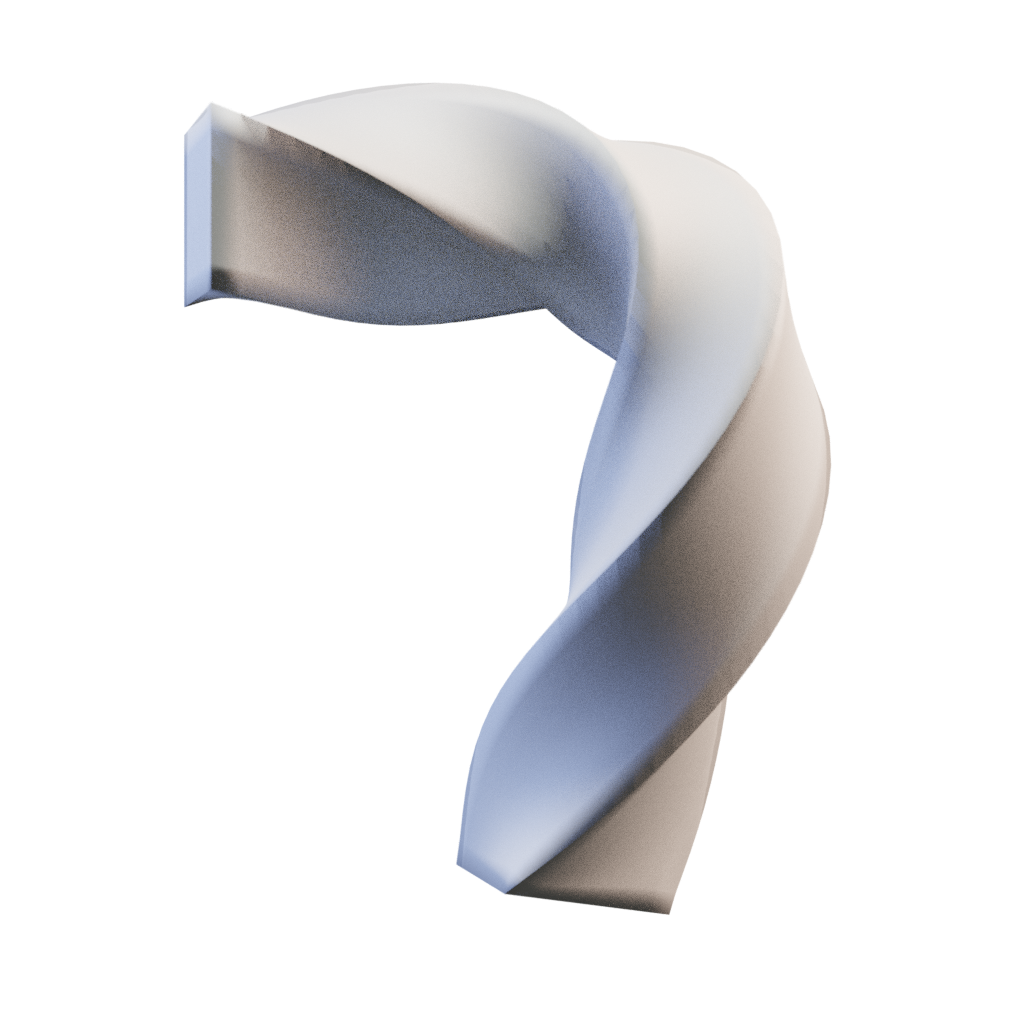}
        \caption{Controlled deformation via twisting and bending a horozontal
            bar with the As-Rigid-As-Possible (ARAP) energy.}
        \label{fig:summary:d}
    \end{subfigure}
    \caption{We present SANM, an open-source framework that automates and
        generalizes the Asymptotic Numerical Method (ANM) to solve symbolically
        represented nonlinear systems via numerical continuation and
        higher-order approximations. We apply SANM to static equilibrium
        problems via continuation on gravity (\ref{fig:summary:a} and
        \ref{fig:summary:b}) and controlled mesh deformation problems with an
        implicit homotopy formulation via continuation on control handles
        (\ref{fig:summary:c} and \ref{fig:summary:d}). SANM takes a symbolic
        representation of the nonlinear system from the user (\cref{code:build})
        and automatically handles all the complexity of applying ANM.
    }
    \label{fig:summary}
\end{teaserfigure}

% vim: tw=80 filetype=tex foldmethod=marker foldmarker=f{{{,f}}} spell spelllang=en

\maketitle

\newcommand{\padeIterReduce}{$1.14\pm0.57$}
\newcommand{\padeNrCases}{36}
\newcommand{\overallSpeedup}{1.91}
\newcommand{\gmeanSpeedup}{1.97}
\newcommand{\gmeanSpeedupGravity}{2.21}
\newcommand{\gmeanSpeedupDeform}{1.41}
\newcommand{\speedupNrCases}{24}
\newcommand{\sparseSolverTimeUsed}{$23.43\%\pm1.56\%$}

\section{Introduction}

Solving nonlinear analytic systems (systems that can be locally described by a
convergent power series) is at the core of many graphics and engineering
applications. Such systems are traditionally solved with Newtonian methods that
essentially use a first or second order local approximation. Newtonian methods
converge quadratically fast when such an approximation is accurate enough, and
the initial guess is sufficiently close. However, these assumptions are often
violated in practice, and the convergence is thus much slower~\citep{
bonnans2006numerical}.

This paper considers solving the system $\vf(\vx)+\vv=\vz$ under a numerical
continuation framework~\citep{ allgower2003introduction}. Here $\vx\in\real^n$
is an unknown vector, $\vf:\real^n\mapsto\real^n$ is an analytic function, and
$\vv$ is a constant.  Given an initial solution $\vxz$ such that $
\vf(\vxz)=\vz$, numerical continuation methods trace the final solution via
solving $\vx(\lambda)$ with $\lambda$ ranging from $0$ to $1$ subject to $\vf(
\vx(\lambda)) + \lambda\vv = \vz$. For example, in the static elasticity
equilibrium problem, we encode the unknown node coordinates in $\vx$, the
mapping from node coordinates to node forces in $\vf(\cdot)$, and the static
external force in $\vv$. By setting $\vxz$ as the rest shape, numerical
continuation corresponds to gradually increasing the external force while
simulating the deformation simultaneously. \cref{fig:armadillo-g} presents an
example.

\begin{figure*}[t]
    \centering
    \includegraphics[width=.95\textwidth]{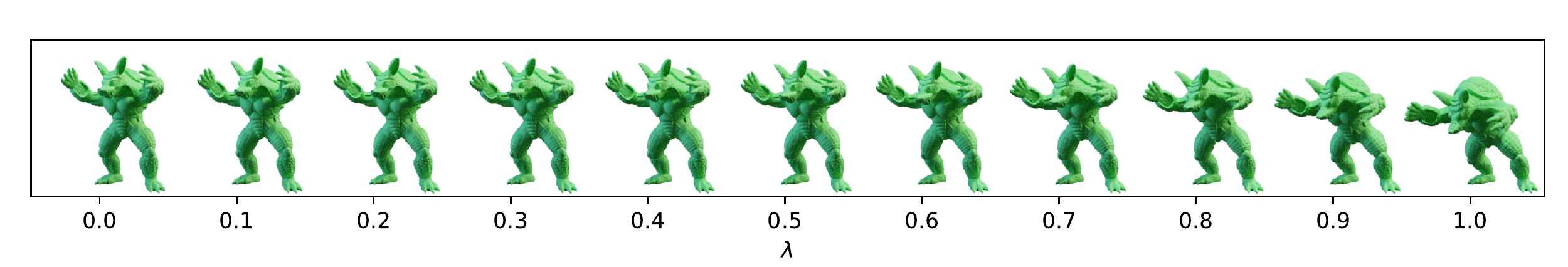}
    \caption{An example of solving gravity equilibrium with numerical
        continuation. Starting from an initial shape $\vx(0)$, we solve
        $\vx(\lambda)$ for $\lambda$ ranging from $0$ to $1$ subject to
        $\vf(\vx(\lambda))+\lambda\V{g}=\vz$ where $\vf(\cdot)$ computes node
        elastic forces given node coordinates, and $\V{g}$ is the per-node
        gravity vector. ANM parameterizes both $\vx$ and $\lambda$ with $a$, and
        approximates $\vx(a)$ and $\lambda(a)$ with power series. ANM allows
        easily computing the intermediate equilibrium states almost for free.
        \label{fig:armadillo-g}
    }
\end{figure*}

The Asymptotic Numerical Method (ANM)~\citep{damil1990new} is a numerical
continuation method that differs fundamentally from Newtonian approaches by
exploring the landscape of the nonlinear system via higher-order approximations.
The numerical continuation formulation $\vf(\vx(\lambda)) + \lambda\vv = \vz$
defines a solution curve $(\vx(\lambda),\, \lambda)$. However, parameterization
of the curve using $\lambda$ can result in ill-conditioned behavior of
$\vx(\lambda)$. Instead, ANM parameterizes both $\vx$ and $\lambda$ with $a$
such that $\vf(\vx(a)) + \lambda(a)\vv = \vz$. ANM approximates the solution
curve with a power series expansion at truncation order $N$: $\vx(a) =
\sum_{i=0}^{N}\vxi a^i$ and $\lambda(a) = \sum_{i=0}^N \lambda_i a^i$.
\citet{cochelin1994path} proposes a continuation technique to compute the final
solution in a stepwise manner. Specifically, they estimate the valid range of
the current approximation as $a_r$ and compute the final solution by iteratively
recomputing the approximation at $\vxz=\vx(a_r)$ and $\lambda_0=\lambda(a_r)$
until $\lambda(a_r)$ reaches $1$.

Compared to the widely used Newtonian methods, ANM is able to determine a local
representation of the solution curve with a larger range of validity in
similar computing time~\citep{cochelin1994asymptotic}, and therefore solves the
nonlinear system in fewer iterations and often less running time.

The core challenge in applying ANM is to solve the expansion coefficients
$\{\vxi\}$ and $\{\lambda_i\}$. The original ANM framework~\citep{ damil1990new}
analytically solves the coefficients for quadratic functions $\vf(\vx) = \vxz +
\V{A}\vx + \trans{\vx}\V{Q}\vx$. \citet{ chen2014asymptotic} deploys ANM on the
inverse deformation problem for 3D fabrication by manually deriving the
coefficient solution for the incompressible neo-Hookean elasticity~\citep{
bonet_wood_2008}. They claim their Taylor coefficient derivation as a major
contribution, which is a difficult and laborious task. They have also shown that
ANM converges up to orders of magnitudes faster than Newtonian solvers.

To date, however, there is no scalable tool that automates the computation of
Taylor coefficients in the general case. The lack of such tools severely limits
the application of ANM to new problems. In this work, we show how to solve the
coefficients $\{\vxi\}$ and $\{\lambda_i \}$ automatically and efficiently for a
symbolically defined function $\vf(\cdot)$. More specifically, we devise
techniques to establish the connection of Taylor expansion coefficients between
$\vx(a)$ and $\qty(\vf \circ \vx)(a)$, which in fact conforms to an affine
relationship for the highest-order term. We analyze a few operators important
for graphics applications, including elementwise analytical functions (such as
power and logarithm), matrix inverse, matrix determinant, and singular value
decomposition. We speed up the system with batch computing that fits naturally
into Finite Element Method (FEM) due to the same computation form shared by all
quadrature points. We present a system, called SANM, as an implementation of our
techniques.

We deploy SANM to solve the forward and inverse static force equilibrium problems
similar to \citet{chen2014asymptotic}. In contrast to their manual derivation
that only works with incompressible neo-Hookean materials, our system allows
easily solving more constitutive models by changing a few lines of code,
including the compressible neo-Hookean model that has a logarithm term and the
As-Rigid-As-Possible energy that involves a polar decomposition. Our
experimental results show that SANM achieves comparable or better performance as
the hand-coded, specialized ANM solver of \citet{chen2014asymptotic}. We are
unaware of efficient alternative methods for the inverse problem other than ANM.
For the forward problem, an alternative is to minimize the total potential
energy, and we show that SANM exhibits better performance than Newtonian energy
minimizers.

We further extend the ANM framework to incorporate implicit homotopy
$\vH(\vx,\,\lambda)=\vz$ where $\vH:\real^{n+1}\mapsto\real^n$ admits a
one-dimensional solution curve. For instance, we formulate the controlled mesh
deformation as an implicit homotopy problem, defined as $\vH(\vx,\,\lambda) =
\vf([\vx;\; \vx_c+ \lambda\V{\Delta}_x])$, where $\vx_c$ corresponds to the
initial location of control handles, $\V{\Delta}_x$ describes their
user-specified movement, and $\vf(\cdot)$ computes the internal elastic force.
The coordinates of unconstrained nodes in the deformed equilibrium state,
denoted by $\vx^*$, are then governed by $\vH(\vx^*,\,1)=\vz$ and can be solved
by continuation on $\lambda$. In our experiments, SANM runs
\gmeanSpeedupDeform~times faster by geometric mean than Newtonian energy
minimization methods. We also demonstrate the robustness and versatility of SANM
by twisting and bending a bar to extreme poses, as shown in
\cref{fig:summary:d}. Note that ANM and Newtonian minimization methods target
different problems and can not replace each other. \cref{sec:meth-discuss}
further discusses their differences.

To summarize, this paper makes the following contributions:
\begin{enumerate}
    \item We devise analytical solutions for Taylor coefficient propagation
        through a few nonlinear operators on which ANM has not been applied,
        including singular value decomposition as a challenging case
        (\cref{sec:single-opr}).
    \item We present a system, SANM, that automatically solves the Taylor
        expansion coefficients for symbolically defined functions (%
        \cref{sec:framework}). SANM greatly reduces programming effort for
        adopting ANM-based methods (\cref{sec:API}). SANM adopts generic
        and FEM-specific optimizations to improve solving efficiency further.
    \item We present a novel continuation algorithm to reduce accumulated
        numerical error and approximation error when solving the equational form
        $\vf(\vx)+\vv=\vz$ (\cref{sec:continue-with-eqn}).
    \item We extend ANM to handle implicit homotopy and apply it to controlled
        mesh deformation problems (\cref{sec:controlled-deform}). Our
        experiments show that SANM often converges faster than a
        state-of-the-art Newtonian energy minimizer. Moreover, the numerical
        continuation framework of SANM directly handles constitutive models that
        do not support inverted tetrahedrons, which would be challenging for
        energy minimization methods due to undefined elastic energy at the
        initial guess.
\end{enumerate}

SANM is available at \url{https://github.com/jia-kai/SANM}.

% vim: tw=80 filetype=tex foldmethod=marker foldmarker=f{{{,f}}} spell spelllang=en

\section{Related Work}

\paragraph{Numerical Optimization:} Numerical optimization has been extensively
studied, and it is closely related to solving nonlinear systems. For example, we
can recast solving $\vf(\vx)=\vz$ as minimizing $g(\vx) = \trans{\vf(\vx)}
\vf(\vx)$ and apply generic minimization methods such as the Levenberg–Marquardt
algorithm. On the other hand, minimizing $f(\vx)$ can often be approached via
solving $\V{\nabla} f(\vx) = \vz$. For controlled mesh deformation problems, the
internal elastic force corresponds to the gradient of the elastic potential
energy with respect to node locations. Therefore, one can either directly solve
a force equilibrium under Dirichlet boundary conditions (as done by SANM) or
minimize the total potential energy to obtain the deformed state.  We review the
development of As-Rigid-As-Possible (ARAP) energy minimization as an example of
improvements on numerical optimizers.  \citet{ sorkine2007rigid} devises a
surface modeling technique by minimizing the ARAP energy via alternating between
fitting the rotations and optimizing the locations. \citet{ chao2010simple}
employs a Newton trust region solver to minimize the ARAP energy. \citet{
shtengel2017geometric} accelerates the convergence by computing a positive
semidefinite Hessian via constructing a convex majorizer for a specific class of
convex-concave decomposable objectives, including the 2D ARAP energy.  \citet{
smith2019analytic} presents analytical solutions for the eigensystems of
isotropic distortion energies to enable easily projecting the Hessians of 2D and
3D ARAP energies to be positive semidefinite to speed up the convergence. Most
optimization methods inherently build on the classic idea of using first or
second order approximations and exploit problem-specific optimization
opportunities. This work targets generic nonlinear solving with numerical
continuation and uses higher-order approximation.

\paragraph{Mesh Deformation:} Mesh deformation control is an important and
widely studied problem in graphics. For animation production that only requires
plausible but not physically accurate results, the simulation performance can be
improved by a variety of approaches such as model analysis \citep{choi2005modal,
kim2009skipping}, skinning \citep{gilles2011frame}, and constraint projection
\citep{bender2014survey, bouaziz2014projective}. For physically predictive
simulations, we need to stick to the formulation derived from continuum
mechanics strictly, and the solver convergence rate is often improved by Hessian
modification~\citep{shtengel2017geometric, kim2020dynamic}. In this paper, we
choose physically accurate elastic deformation as our target application. We
approach the problem by solving a nonlinear system that encodes force
equilibrium constraints.

\paragraph{Numerical Continuation Methods:} Classic numerical continuation
methods include the predictor-corrector method and the piecewise-linear method.
\citet{allgower2003introduction} provides an introduction to this topic. The
basic idea, which is to follow a solution trajectory by taking small steps, has
become popular in many applications such as motion planning~\citep{
yin2008continuation,duenser2020robocut}, MRI reconstruction~\citep{
trzasko2008highly}, and drawing assistance~\citep{ limpaecher2013real}. These
works typically choose a fixed step size or adopt a problem-specific step size
schedule in the predictor and use classic first or second order solvers as the
corrector. By contrast, asymptotic numerical methods use a higher-order
approximation as the predictor without needing a corrector and adaptively choose
the step size according to how well the predictor approximates the system.

\paragraph{Asymptotic Numerical Methods:} ANM has been applied to solve
engineering problems in different domains, including buckling analysis \citep{
azrar1993asymptotic, boutyour2004bifurcation}, vibration analysis \citep{
azrar2002non, daya2001numerical}, shell and rod simulation \citep{
zahrouni1999computing,  lazarus2013continuation}, and inverse deformation
problems \citep{chen2014asymptotic}. ANM assumes a quadratic system.  An
improvement over the standard ANM framework is to increase the range of validity
of the approximation via imposing heuristics on the function behavior, such as
replacing the power series with a Pad{\'e} representation \citep{
najah1998critical, cochelin1994asymptoticpade, elhage2000numerical}. When
adapting ANM to new problems, one typically needs to recast their specific
problems into quadratic forms by introducing auxiliary variables and deriving
the expansions manually \citep{ guillot2019generic}.  \citet{
abichou2002asymptotic} presents a review on adapting ANM for a few nonlinear
functions.

Few attempts have been made to automate ANM to handle general nonlinearities.
Notably, \citet{charpentier2008diamant} proposes an automatic differentiation
framework, called Diamant, that computes the expansion coefficients by Taylor
coefficient propagation via computing higher-order derivatives of the operators.
Unfortunately, the Diamant approach is not readily applicable to mesh
deformation problems due to the difficulty in computing higher-order derivatives
of certain matrix functions such as matrix inverse or determinant used in the
constitutive models. Moreover, Diamant is not designed with high-performance
computing in mind. It only works with scalar variables, does not take advantage
of the structural sparsity in FEM problems, and is only evaluated on small-sized
problems. \citet{lejeune2012object} incorporates the Diamant approach into an
object-oriented solver to automate ANM. By contrast, SANM natively works with
multidimensional variables and is accelerated with batch computing for
large-scale FEM problems. SANM also implements a generic framework for computing
the expansion coefficients, which is not limited to the higher-order derivative
approach of Diamant and is capable of handling challenging matrix functions.

% vim: tw=80 filetype=tex foldmethod=marker foldmarker=f{{{,f}}} spell spelllang=en

\section{ANM Background}
This section introduces the asymptotic numerical method. We begin with a toy
example of a geometry problem and then formally describe ANM. We first define
the notations used in this paper in \cref{tab:notation}.

\begin{table}[hbt]
    \centering
    \begingroup
    \setlength{\tabcolsep}{2pt}
    \caption{Notation definition}
    \label{tab:notation}
    \begin{tabular}{lp{0.78\columnwidth}}
        $f$, $x$ & Scalars or scalar-valued functions \\
        $\vf$, $\vx$ & Vectors or vector-valued functions \\
        $f_i$, $x_i$ & A scalar in the vector at given index. For a function
            $f(\cdot)$, we also use $f_i$ to represent its $\nth{i}$ Taylor
            coefficient, and similarly for $\vf(\cdot)$ vs $\vfi$ and
            $\vF(\cdot)$ vs $\vFi$. \\
        $\vfi$, $\vxi$ & A vector in an array of vectors \\
        $\vF$, $\vX$ & Matrices or matrix-valued functions \\
        $\V{F_i}$, $\V{X_i}$ & A matrix in an array of matrices \\
        $X_{ij}$ & A coefficient in the matrix at given row and column \\
        $\matrow{\vX}{i}$, $\matcol{\vX}{j}$ & The vectors corresponding to the
             $\nth{i}$ row or the $\nth{j}$ column in matrix $\vX$ \\
        $ \rms(\vx) $ & Root-mean-square of $\vx \in \real^n$:
            $\rms(\vx) = \sqrt{\nicefrac{\trans{\vx}\vx}{n}}$ \\
        $\vect(\vX)$ & Flatten a matrix $\vX$ into a column vector by
            concatenating the columns in $\vX$ \\
        $\norm{\vX}$ & Frobenius norm of the matrix $\vX$, defined as
            $\sqrt{\trans{\vect(\vX)}{\vect(\vX)}}$ \\
        $\diag(\vX)$ & A vector containing the diagonal coefficients of $\vX$ \\
        $o(x)$ & The little-o notation: $y=o(x)$ if $\nicefrac{y}{x} \to 0$.
    \end{tabular}
    \endgroup
\end{table}

\subsection{A Circle-ellipse Intersection Problem}
% f{{{
\begin{figure}[h]
    \centering
    \includegraphics[width=.8\columnwidth]{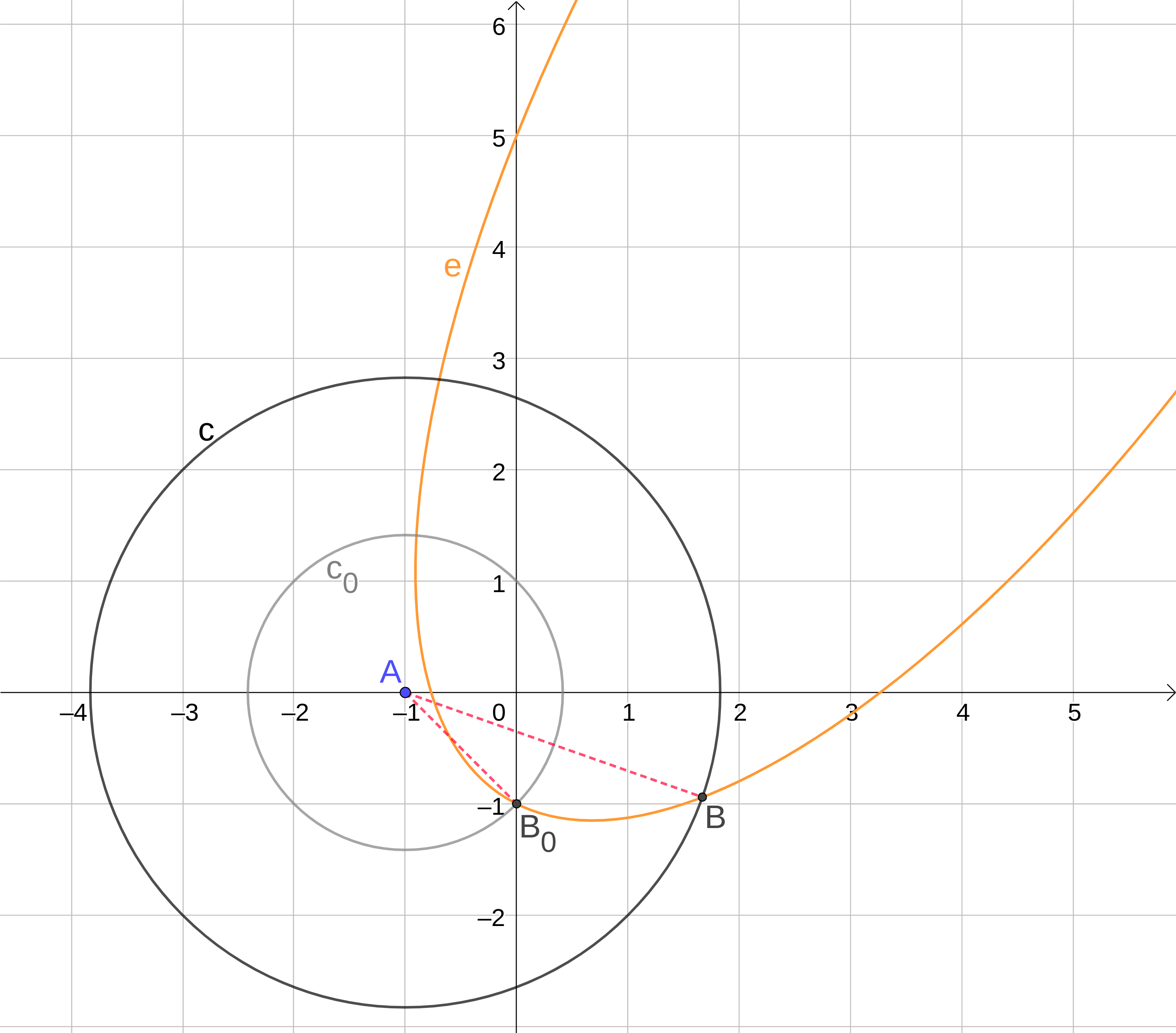}
    \caption{Our circle-ellipse intersection example problem. ANM uses
        polynomials to approximate the trace of intersection points between the
        ellipse and concentric circles of radius from $\sqrt{2}$ to $\sqrt{8}$,
        which is the arc $\wideparen{\V{B_0}\V{B}}$.}
    \label{fig:geo-example}
\end{figure}

We illustrate ANM with a toy problem that asks for the intersection $\V{B}$ of
an ellipse $e$ and a circle $c$ as shown in \cref{fig:geo-example}. The ellipse
intersects the y-axis at $\V{B_0}=(0,\,-1)$. The circle is centered at
$A=(-1,\,0)$ and its radius is $|\V{A}\V{B}|=\sqrt{8}$. To solve the problem
with ANM, we first choose a continuation scheme. We start with a smaller
concentric circle $c_0$ with radius $|\V{A}\V{B_0}|=\sqrt{2}$ and continuously
trace its intersection with $e$ while increasing its radius from $\sqrt{2}$ to
$\sqrt{8}$. Note that the continuation scheme is problem-specific. While there
can be many choices, practical problems typically admit a ``natural'' choice,
such as the external force in static equilibrium problems.

Formally, our goal is to solve $(x,\,y)$ such that $f_e(x,\,y) = f_c(x,\,y)=0$,
where $f_e(\cdot)$ and $f_c(\cdot)$ describe the ellipse $e$ and the circle $c$
respectively:
\begin{align}
    \arraycolsep=.5ex
    \begin{array}{rl}
        f_e(x,\,y) =& 2x^2 - 5x + y^2 - 4y - 2xy - 5 \\
        f_c(x,\,y) =& (x+1)^2 + y^2 - 8
    \end{array}
\end{align}

ANM introduces a variable $\lambda\in[0,\,1]$ to represent the continuation.
ANM traces the solution curve starting at $(x_0,\,y_0)$ via varying $\lambda$
from $0$ to $1$ while keeping the following equations satisfied:
\begin{align}
    \arraycolsep=.5ex
    \begin{array}{rl}
        f_e(x,\,y) &= 0\lambda \\
        f_c(x,\,y) + 6 &= 6\lambda \\
        (x_0,\,y_0) &= \V{B_0} = (0, \,-1)
    \end{array}
    \label{eqn:geo-example-sys}
\end{align}

Geometrically, ANM continuously solves the intersection between the ellipse and
a concentric circle with radius $\sqrt{2+6\lambda}$. ANM parameterizes the
solution curve by a variable $a$ and approximates $x(a)$, $y(a)$, and
$\lambda(a) $ with polynomial expansions at truncation order $N$, with
coefficients $\{x_k\}$, $\{y_k\}$, and $\{\lambda_k\}$ to be solved:
\begin{align}
    \arraycolsep=.5ex\def\arraystretch{1.5}
    \begin{array}{rl}
        x(a) &= 0 + \sum_{k=1}^N x_k a^k \\
        y(a) &= -1 + \sum_{k=1}^N y_k a^k \\
        \lambda(a) &= 0 + \sum_{k=1}^N \lambda_k a^k \\
    \end{array}
    \label{eqn:geo-example-expand}
\end{align}

\begin{figure}[t]
    \centering
    \includegraphics[width=.8\columnwidth]{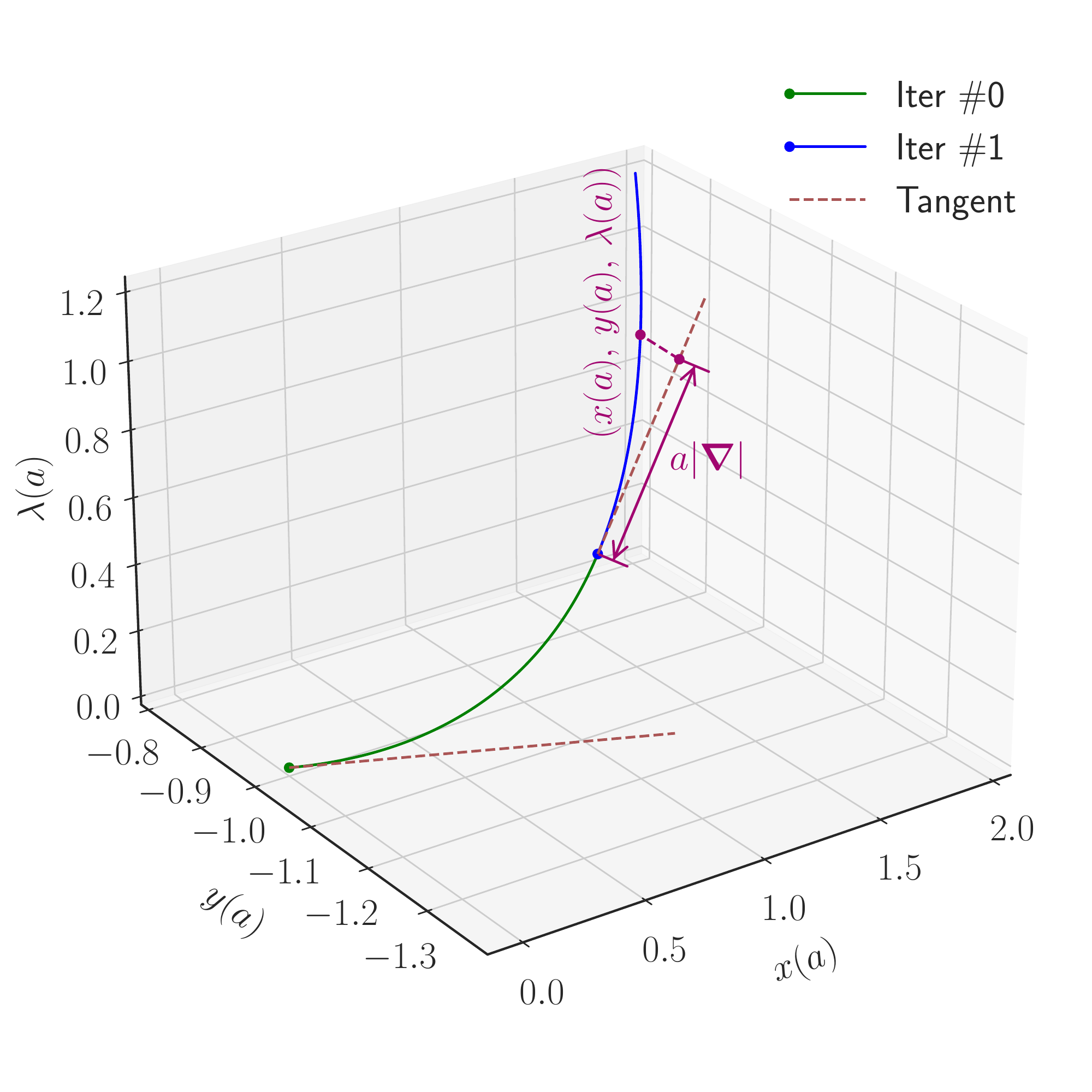}
    \caption{ANM solution curve for the circle-ellipse intersection problem. ANM
        solves the problem in two iterations. The path parameter $a$ is
        identified by the projection of a point $(x(a),\,y(a),\,\lambda(a))$
        onto the tangent direction. The length of the projection is
        $a|\V{\nabla}|$ where $\V{\nabla}=(x'(0),\,y'(0),\,\lambda'(0))$ is the
        gradient at the beginning of the current iteration.
    }
    \label{fig:geo-example-anm}
\end{figure}

We iteratively solve the coefficients by introducing the lower-order terms in
\eqnref{geo-example-expand} into \eqnref{geo-example-sys}. We start with $(x_1,
\, y_1,\, \lambda_1)$ and introduce $x=x_1a$, $y=-1+y_1a$, and $\lambda=
\lambda_1 a$ into \eqnref{geo-example-sys}:
\begin{align}
    \arraycolsep=.5ex
    \begin{array}{rll}
        f_e(x_1a,\,y_1a - 1) &= -(3x_1 + 6y_1)a + o(a) &= 0 \\
        f_c(x_1a,\,y_1a - 1) + 6 &= (2x_1-2y_1)a + o(a) &= 6\lambda_1a \\
    \end{array}
    \label{eqn:geo-example-c1}
\end{align}

We obtain two linear constraints on the three unknowns by equating the
coefficient of $a$ in \eqnref{geo-example-c1}. Let $\V{u}(a)=[x(a);\;y(a);\;
\lambda(a)]$ denote the solution curve. ANM further identifies the path
parameter $a$ as the pseudo-arclength that is the projection of the path along
its tangent direction as shown in \cref{fig:geo-example-anm}, specifically
$a=\trans{\qty(\V{u}(a)-\V{u}(0))} \V{u}'(0)$, which provides the third
constraint for a full-rank system:
\begin{align}
    \arraycolsep=.5ex
    \left\{
    \begin{array}{rl}
        -3x_1 - 6y_1 &= 0 \\
        2x_1-2y_1- 6\lambda_1 &= 0 \\
        x_1^2 + y_1^2 + \lambda_1^2 &= 1
    \end{array}
    \right.
    \label{eqn:geo-example-c1full}
\end{align}

We require $\lambda_1$ to be positive so that $\lambda(a)$ is a locally
increasing function at $0$, and the solution of \eqnref{geo-example-c1full} is
$x_1=\nicefrac{2}{\sqrt{6}}$, $y_1=-\nicefrac{1}{\sqrt{6}}$, and
$\lambda_1=\nicefrac{1}{\sqrt{6}}$. We then solve $(x_2,\,y_2,\,\lambda_2)$ by
equating the coefficients of $a^2$ in $f_e(x_2a^2 + \nicefrac{2}{\sqrt{6}}a,
\,y_2a^2 -\nicefrac{1}{\sqrt{6}}a -1) = 0$ and $f_c(x_2a^2 +
\nicefrac{2}{\sqrt{6}}a,\,y_2a^2 -\nicefrac{1}{\sqrt{6}}a -1) = 6(\lambda_2a^2 +
\nicefrac{1}{\sqrt{6}}a)$. Expanding the equations results in two linear
constraints for the three unknowns. The third pseudo-arclength constraint is
$x_1x_k + y_1y_k + \lambda_1\lambda_k = 0$ for $k \ge 2$. Repeating this step,
we can solve the coefficients $\{x_k\}$, $\{y_k\}$, and $\{\lambda_k\}$ that
define the polynomials $x(a)$, $y(a)$, and $\lambda(a)$. As will be shown in
\cref{sec:linearity}, the equations for all $(x_k,\,y_k,\,\lambda_k)$ are
linear, and the coefficients in these equations are the gradients of
$f_e(\cdot)$ and $f_c(\cdot)$ evaluated at $(x_0,\,y_0)$.

We then estimate $a_r$, the range of validity of the polynomial approximations
$x(a)$, $y(a)$, and $\lambda(a)$. We iteratively compute a new approximation at
$(x(a_r),\,y(a_r))$ to extend the solution curve until $\lambda(a_r)\ge1$, and
we compute the final solution $(x(a^*),\,y(a^*))$ with $a^*=\lambda^{-1}(1)$.
ANM with truncation order $N=20$ is able to find the circle-ellipse intersection
with a residual (defined as $\sqrt{\nicefrac{\qty(f_e^2(\vx) + f_c^2(\vx))}{2}}
$) of $2\times10^{-6}$ in two iterations. \cref{fig:geo-example-anm} visualizes
the parametric solution curve. The continuation formulation presented in
\cref{sec:continue-with-eqn} further reduces the residual to $7\times10^{-9}$.

An alternative is the Newton-Raphson method, which iteratively computes
$\V{x_{n+1}} = \V{x_n} - \V{J}(\V{x_n})^{-1}\vf(\V{x_n})$, where $\vf(\vx) =
[f_e(\vx);\; f_c(\vx)]$ and $\V{J}(\vx)$ is the Jacobian of $\vf$ evaluated at
$\vx$. Starting at $\V{x_0}=(0, \,-1)$, the Newton-Raphson method needs four
evaluations of the Jacobian to converge to a solution with a residual of
$8\times10^{-5}$.

% f}}}

\subsection{ANM Overview}
% f{{{
ANM aims to solve the nonlinear system $\vf(\vx)+\vv=\vz$ with numerical
continuation, where $\vf:\real^n \mapsto \real^n$ is an analytic function, and
$\vv$ is a constant. Starting from an initial solution $(\vxz,\,\lambda_0)$ such
that $\vf(\vxz) + \lambda_0\vv=\vz$, continuation methods compute an
approximation to trace the nearby solution curve of $\vf(\vx)+\lambda\vv=\vz$.
In the general case, the curve may not be well-conditioned under the
parameterization with respect to $\lambda$, and it is preferable to consider an
arclength parameterization $\vx(a)$ and $\lambda(a)$ where $a$ measures the
arclength or pseudo-arclength~\citep{ allgower2003introduction}.

ANM approximates $\vx(a)$ and $\lambda(a)$ by Taylor expansion at truncation
order $N$ such that $\vf(\vx(a)) + \lambda(a)\vv$ should be sufficiently close
to zero for small values of $a$:
\begin{align}
    \label{eqn:series-def}
    \arraycolsep=.5ex\def\arraystretch{1.5}
    \begin{array}{rl}
        \vx(a) &=  \sum_{i=0}^{N}\vxi a^i \\
        \lambda(a) &=  \sum_{i=0}^N \lambda_i a^i \\
        \text{s.t. } & \norm{\vf(\vx(a)) + \lambda(a)\vv}=o(a^N)
    \end{array}
\end{align}

We require $\lambda_1>0$ so that $\lambda(a)$ is locally increasing, and the
algorithm makes progress. We then estimate the range of validity $a_r$ such that
$\vx(a)$ and $\lambda(a)$ are good approximations of the solution when
$|a|<a_r$. If $\lambda(a_r) \ge 1$, we can solve $a^*$ such that
$\lambda(a^*)=1$ and compute the final solution $x^* = \vx(a^*)$. Otherwise,
when $\lambda(a_r) < 1$, we recompute the power series approximation starting at
$\vxz=\vx(a_r)$ and $\lambda_0=\lambda(a_r)$ and repeat the above steps.

A simple method to estimate $a_r$, as suggested by \citet{cochelin1994path},
builds on the idea that within the range of validity, different orders of
approximation should behave similarly:
\begin{align}
    \frac{\norm{\vx(a)_{\text{order $N$}} -
    \vx(a)_{\text{order $N-1$}}}}{\norm{\vx(a)_{\text{order $N$}} - \vxz}}
    < \epsilon
\end{align}
which leads to an approximation
\begin{align}
    \label{eqn:rov-taylor}
    a_r \approx \qty( \epsilon\frac{\norm{\V{x_1}}}{\norm{\V{x_N}}} )^{
    \frac{1}{N-1}}
\end{align}

The equation $\lambda(a^*)=1$ can be solved by a univariate polynomial root
finding algorithm such as Brent's method~\citep{ brent2013algorithms}.

The remaining part of completing the ANM algorithm is to solve the coefficients
$\{\vxi\}$ and $\{\lambda_i\}$ efficiently, which is a core contribution of this
paper.

We solve the coefficients $\{\vxi\}$ and $\{\lambda_i\}$ iteratively. Assume we
are at the $\nth{k}$ iteration, where $\{\vxz,\ldots,\V{x_{k-1}}\}$ and $\{
\lambda_0,\,\ldots,\,\lambda_{k-1}\}$ have been solved. With $\V{x_k}\in\real^n$
and $\lambda_k\in\real$ currently unknown, we have the equation:
\begin{align}
    \vf(\vx(a))) + \lambda(a)\vv \approx
    \vf\qty(\sum_{i=0}^k \vxi a^i) +
    \qty(\sum_{i=0}^k \lambda_i a^i)\vv \approx \vz
\end{align}

Assume $\{\vfi\}$ are the Taylor coefficients of $\vf(\vx(a))$:
\begin{align}
    \vf(\vx(a)) \approx \vf\qty(\sum_{i=0}^k \vxi a^i) =
        \sum_{i=0}^k \vfi a^i + o(a^k)
    \label{eqn:xk-fk-taylor}
\end{align}

As will be shown in \cref{prop:linear}, for an analytic function $\vf(\cdot)$,
there is an affine relationship between $\vfk$ and $\vxk$, specifically $\vfk =
\vP(\vxz)\vxk + \vq(\vxz, \ldots, \V{x_{k-1}})$, where
$\vP(\vxz)\in\real^{n\times n}$ is the slope matrix and $\vq(\cdot)\in\real^n$
is the bias vector. By introducing this relationship into the original equation
and requiring the coefficient of $a^k$ to be zero, we obtain a linear system
that restricts $\vxk$ and $\lambda_k$:
\begin{align}
    \label{eqn:xk-lambdak-f}
    \vP(\vxz)\vxk + \lambda_k \vv = -\vq(\vxz,\, \ldots,\, \V{x_{k-1}})
\end{align}

However, the system has rank $n$ but there are $n+1$ unknowns because the
curve behavior with respect to its parameter $a$ is not fully constrained. To
obtain a full rank system, \citet{cochelin1994path} proposes to identify the
path parameter $a$ as the pseudo-arclength, similar to other numerical
continuation methods \citep{ allgower2003introduction}. Pseudo-arclength
approximates the arclength of a curve by projecting it onto the tangent space,
which constitutes the following constraint:
\begin{align}
    \label{eqn:a-is-proj}
    a = \trans{(\vx(a)-\vxz)}\V{x'}(0) +
        (\lambda(a)-\lambda_0) \lambda'(0)
\end{align}
Introduce \eqnref{series-def} into \eqnref{a-is-proj}:
\begin{align}
    \label{eqn:xk-lambdak-proj}
    \trans{\V{x_i}}\V{x_1} + \lambda_i\lambda_1 = \indicator{i=1}
\end{align}

The unknowns $\vxk$ and $\lambda_k$ can be solved by combining
\eqnref{xk-lambdak-f} and \eqnref{xk-lambdak-proj}. The solution is unique if we
further require $\lambda'(0) = \lambda_1 > 0$.

% f}}}

\subsection{Linearity Between Taylor Coefficients}
% f{{{
\label{sec:linearity}

Before proving the linearity between $\vfk$ and $\vxk$ in \eqnref{xk-fk-taylor},
we introduce an auxiliary definition:

\begin{definition}
    Define $\coeffof{k}{f(a)}$ to be the coefficient of $a^k$ in the Taylor
    expansion of an analytic function $f:\real\mapsto\real$ such that:
    \begin{align*}
        f(a) = \sum_{i\ge0} \coeffof{i}{f(a)} a^i
    \end{align*}
\end{definition}
\begin{prop}
    \label{prop:linear}
    Let $f:\real^n\mapsto\real$ be an analytic function, and $\vxz,\, \ldots,\,
    \V{x_{k-1}}$ known real-valued $n$-dimensional vectors. Assume coefficients
    $f_0,\, \ldots,\, f_k$ satisfy:
    \begin{align*}
        f\qty(\sum_{i=0}^k \vxi a^i) = \sum_{i=0}^k f_ia^i + o\qty(a^k)
    \end{align*}
    Then we have $f_k = \trans{\vp(\vxz)}\vxk + \vq(\vxz,\, \ldots,\,
    \V{x_{k-1}})$.  Specifically, $\vp(\vxz) = \grad{f}(\vxz)$ and $\vq(\vxz,\,
        \ldots,\,\V{x_{k-1}}) = \coeffof{k}{f\left(\sum_{i=0}^{k-1}\vxi a^i
        \right)}$.
\end{prop}
\begin{proof}
    We prove the univariate case for the simplicity of the notations. Our
    argument also applies to multivariate functions by using the multivariate
    Taylor theorem.

    Let $g(t) = f(x_0 + t) = \sum_{i\ge0} g_i t^i$ where $g_i=\frac{1}{i!}
    \eval{\dv[i]{g}{t}}_{t=0}$ is the Taylor expansion coefficient and $g_0=f_0
    = f(x_0)$. With $t=\sum_{i=1}^kx_ia^i$, we rewrite
    \begin{align*}
        f\qty(\sum_{i=0}^k x_i a^i) = g\qty(\sum_{i=1}^kx_ia^i) =
            g_0 + \sum_{i\ge1} g_i \qty(\sum_{j=1}^k x_ja^j)^i
    \end{align*}
    To compute $f_k$, we consider the terms that contribute to $a^k$ in the
    expansion of the right side hand:
    \begin{enumerate}
        \item For terms that contain $x_k$ and contribute to $f_k$, it must
            contain $x_ka^k$ and no other $x_ia^i$ for which $i>0$.  There is
            only one such term, which is $g_1x_ka^k$. The slope $p$ for which
            $f_k = px_k+q$ is thus $p=g_1=f'(x_0)$.
        \item The bias $q$ consists of terms that do not contain $x_k$, which
            can be computed by treating $x_k$ as zero, or equivalently removing
            $x_ka^k$ from $t$:
            \begin{align*}
                q = \coeffof{k}{f\qty(\sum_{i=0}^{k-1}x_ia^i)}
            \end{align*}
    \end{enumerate}
\end{proof}

\paragraph{Remarks:} We have discussed the case of scalar functions. For a
vector function $\vf:\real^n\mapsto\real^m$, the slope matrix $\vP$ is its
Jacobian. Note that $\vp(\vxz)$ only depends on the initial point $\vxz$ and
remains constant for all orders. This allows us to factorize the coefficient
matrix only once to solve all the terms. The result of \cref{prop:linear} is not
new.  For example, it is a direct consequence of Fa{\`a} di Bruno's formula
\citep{ roman1980formula}. Here we have presented a simple proof based on
elementary calculus.
% f}}}

\subsection{Continuation with Pad{\'e} Approximation}
% f{{{
\label{sec:pade}

The original ANM approximates $\vx(a)$ and $\lambda(a)$ with Taylor expansions.
It has been shown that replacing the Taylor expansions with Pad{\'e}
approximations results in a larger range of validity and thus fewer iterations.

For a scalar function $f:\real\mapsto\real$, its Pad{\'e} approximation of order
$M+N$ approximates the function with a ratio of two polynomials, $P_N(x)$ and
$Q_M(x)$ of degrees $N$ and $M$, respectively, such that $f(x) =
\frac{P_N(x)}{Q_M(x)} + o(x^{N+M})$. The polynomials $P_N(x)$ and $Q_M(x)$ are
determined by the first $N+M+1$ Taylor coefficients of $f$ via a set of linear
constraints up to a common scaling factor. Although the Pad{\'e} approximation
is constructed from the information contained in the Taylor expansion, for many
functions, it has a larger range of validity than the Taylor series, and it
sometimes gives meaningful results even when the radius of convergence of the
Taylor series is strictly zero \citep{ basdevant1972pade}.

\citet{cochelin1994asymptoticpade} proposes to construct a Pad{\'e}
approximation for the vector function $\vx(a)=\sum_{i=0}^N\vxi a^i$ in the form:
\begin{align}
    \vP_N(a) &= \vxz + \sum_{i=1}^{N-1}
        \frac{D_{N-1-i}(a)}{D_{N-1}(a)} \vxi a^i
        \nonumber \\
    D_k(a) &=  \sum_{i=0}^k d_i a^i
\end{align}
\citet{najah1998critical} presents the process of computing the coefficients
$\{d_i\}$ from $\{\vxi\}$, which first orthonormalizes $\{\vxi\}$ and then
solves $\{d_i\}$ based on the principle that the Pad{\'e} approximation should
share the same lower-order Taylor coefficients. We omit the details here.

We follow the techniques presented in \citet{ elhage2000numerical} to determine
the range of validity $a_p$ for this Pad{\'e} approximation:
\begin{align}
    \frac{\norm{P_N(a_p)-P_{N-1}(a_p)}}{\norm{P_N(a_p)-P_N(0)}} < \epsilon
\end{align}
The number $a_p$ is sought via bisection in the range $(a_r, r)$, where $a_r$ is
the range of validity of the Taylor series determined by \eqnref{rov-taylor},
and $r$ is the smallest positive real root of $D_{N-1}(a)$ that can be found by
numerical methods such as Bairstow algorithm \citep{golub1967generalized}. Note
that $P_{N-1}$ is not the first $N-1$ terms in $P_N$ but the Pad{\'e}
approximation computed from $\{\vxz,\,\ldots,\,\V{x_{N-1}}\}$. For all the
\padeNrCases~test cases presented in \cref{sec:application}, Pad{\'e}
approximation uses \padeIterReduce~fewer iterations than the original ANM
formulation on average.

% f}}}

% vim: tw=80 filetype=tex foldmethod=marker foldmarker=f{{{,f}}} spell spelllang=en

\section{The SANM Framework}
\label{sec:framework}

This section presents the overall design of SANM and its two novel extensions
over ANM: the handling of implicit homotopy $\vH(\vx,\,\lambda)=\vz$ and a
formulation to reduce accumulated error when solving the equational form
$\vf(\vx) + \vv=\vz$. Note that this paper deals with both computing graphs and
mesh networks. We use \emph{vertex} to refer to a vertex in a graph and
\emph{node} to refer to a node in a polygon mesh.

\subsection{Coefficient Propagation on the Computing Graph}
% f{{{
\label{sec:coeff-prop}

\begin{algorithm}[t]
\caption{Taylor coefficient solver in SANM}
\label{algo:comp-coeff}
\begin{algorithmic}
    \State {\bfseries Input:} Computing graph $G=(V_o \cup V_v,\, E)$ \\
        \hspace{3em} that represents $\vH: \real^{n+1}\mapsto\real^n$
    \State {\bfseries Input:} Initial value $\vxz\in\real^n$ and $\lambda_0\in
        \real$ such that $\vH(\vxz,\,\lambda_0)=\vz$
    \State {\bfseries Input:} Truncation order $N$
    \State {\bfseries Output:} Taylor coefficients $\{\vxi\}$ and
        $\{\lambda_i\}$ such that \\
        \hspace{3em} $\norm{\vH\qty(\sum_{i=0}^N \vxi a^i,\,
            \sum_{i=0}^N \lambda_i a^i)} = o(a^N)$
    \vspace{1em}

    \State Compute $\vP \leftarrow \pdv{\vH}{\vx}\/(\vxz,\,\lambda_0)$ by
        reverse mode AD on $G$
    \State Meanwhile, record the Jacobians $\V{p^u}$ for each operator
        $u\in V_o$
    \State Meanwhile, record $\vv \leftarrow
        \pdv{\vH}{\lambda}\/(\vxz,\,\lambda_0)$
    \State Factorize $\vP$
    \For{$k \leftarrow 1$ to $N$}
        \State Compute the biases $\V{q_k^u}$ for $u\in V_o$ in
            topological order, \\
            \hspace{3em} given $\vxz,\, \lambda_0,\, \ldots,\,
            \V{x_{k-1}},\,\lambda_{k-1}$ and the rules in \cref{sec:single-opr}
        \State Compute $\V{q_k}$ by combining affine transformations
            $(\V{p^u},\,\V{q_k^u})$ \\
            \hspace{3em} for $u\in V_o$ in topological order
        \State Solve $\vxk$ and $\lambda_k$ according to
            \eqnref{xk-lambdak-proj} and \eqnref{xk-lambdak-H}:
            \begin{align*}
                \left\{
                \begin{array}{l}
                    \vP \vxk + \lambda_k\vv = -\V{q_k} \\
                    \trans{\vxk}\V{x_1} + \lambda_k\lambda_1 = \indicator{k=1}
                \end{array}
                \right.
            \end{align*}
    \EndFor
\end{algorithmic}
\end{algorithm}

SANM represents a nonlinear function $\vf(\cdot)$ as a directed acyclic
bipartite computing graph composed of predefined operators: $G=(V_o \cup
V_v,\,E)$. The user builds the graph to specify the nonlinear function in their
problem symbolically. The vertex set $V_o$ corresponds to a subset of predefined
operators offered by SANM: for $u \in V_o$, there is a function $f_u$ that
defines the corresponding computing.  Each operator takes a fixed number of
inputs, where each input can be a higher-order tensor. For example, matrix
inverse takes one matrix input, and vector addition takes two vector inputs. The
vertex set $V_v$ represents the variables. For a variable $v \in V_v$ and an
operator $u \in V_o$, an edge $(v,\,u) \in E$ if the user passes $v$ as an input
of the function $f_u$. An edge $(u,\,v) \in E$ if $v$ is an output of the
function $f_u$. Many successful symbolic-numerical systems, such as Theano
\citep{ 2016arXiv160502688short} and Tensorflow \citep{abadi2016tensorflow},
have adopted computing graphs to represent user-defined functions.

Besides working with the original ANM formulation $\vf(\vx)+\lambda\vv=\vz$,
SANM also supports solving the more general implicit homotopy: $\vH(\vx,\,
\lambda)=\vz$. Implicit homotopy defines a curve $\vx(\lambda)$. Our goal is to
solve $\vx(1)$ given an initial value $\vx(0)$.

\cref{prop:linear} implies that the $\nth{k}$ order expansion of the implicit
homotopy can be written as:
\begin{multline}
    \vH\qty(\sum_{i=0}^{k}\vxi a^i, \sum_{i=0}^{k}\lambda_i a^i) =
        \sum_{i=0}^{k-1}\V{H_i}a^i + \\
        \qty(
            \pdv{\vH}{\vx}\/(\vxz,\,\lambda_0) \cdot \vxk +
            \pdv{\vH}{\lambda}\/(\vxz,\,\lambda_0) \cdot \lambda_k +
            \V{q_k}
        ) a^k + \cdots
\end{multline}
By requiring the coefficient of $a^k$ to be zero, we obtain the equation
\begin{align}
    \label{eqn:xk-lambdak-H}
    \pdv{\vH}{\vx}\/(\vxz,\,\lambda_0) \cdot \vxk +
        \pdv{\vH}{\lambda}\/(\vxz,\,\lambda_0) \cdot \lambda_k +
        \V{q_k} = \vz
\end{align}

Given a user-defined function $\vH(\cdot)$ and initial values $\vH(\vxz,\,
\lambda_0)=\vz$, SANM computes the slope matrix $\vP = \frac{\partial\vH}{
\partial\vx}(\vxz,\,\lambda_0)$ via reverse mode Automatic Differentiation (AD)
\citep{atilim18automatic}. It then iteratively computes the biases for each
order and solves the Taylor coefficients according to \eqnref{xk-lambdak-proj}
and \eqnref{xk-lambdak-H}. The bias $\V{q_k}$ is computed by merging the affine
transformations of individual operators in $G$. \cref{algo:comp-coeff}
summarizes this process.

% f}}}

\subsection{Continuation in SANM}
% f{{{

\begin{algorithm}[t]
\caption{Continuation framework in SANM}
\label{algo:continue}
\begin{algorithmic}
    \State {\bfseries Input:} An analytic function
        $\vH: \real^{n+1}\mapsto\real^n$
    \State {\bfseries Input:} Initial value $\vxz\in\real^n$ and $\lambda_0\in
        \real$ such that $\vH(\vxz,\,\lambda_0)=\vz$
    \State {\bfseries Input:} Target location of the homotopy $\lambda_t >
        \lambda_0$
    \State {\bfseries Output:} The solution $\vxs \in \real^n$ such that
        $\vH(\vxs,\,\lambda_t) = \vz$
    \vspace{1em}

    \State $k \leftarrow 0$
    \While{$\lambda_k < \lambda_t$}
        \State Compute Taylor coefficients $\{\V{x_{k,i}}\}$ and
            $\{\lambda_{k,i}\}$ \\
            \hspace{3em} using \cref{algo:comp-coeff}
        \State Compute the corresponding Pad{\'e} approximation $P_N(a)$
        \State Compute $a_r$ and $a_p$, the range of validity \\
            \hspace{3em} for these two approximations
        \State $a_m \leftarrow \max(a_r,\, a_p)$
        \State Use Taylor or Pad{\'e} to approximate $\V{f_x}(a)$ and
            $f_\lambda(a)$, \\
            \hspace{3em} depending on which one achieves the range $a_m$
        \State Solve $a'$ such that $f_\lambda(a') = \min(f_\lambda(a_m), \,
            \lambda_t)$
        \State $(\V{x_{k+1}},\,\lambda_{k+1}) \leftarrow
            (\V{f_x}(a'),\,f_\lambda(a'))$
        \State $k \leftarrow k + 1$
    \EndWhile
    \State $\vxs \leftarrow \vxk$
\end{algorithmic}
\end{algorithm}

SANM aims to solve $\V{x^*}$ such that $\vH(\V{x^*}, \, \lambda_t) = \vz$ given
initial values $\vH(\vxz,\, \lambda_0)=\vz$. We typically have $\lambda_0=0$
and $\lambda_t=1$. Note that we have changed the meaning of $\vxk$ and
$\lambda_k$ to indicate the initial values at the $\nth{k}$ iteration rather than
the $\nth{k}$ Taylor coefficient as in \cref{sec:coeff-prop}.

Similar to the standard ANM procedure, at the $\nth{k}$ iteration, SANM uses
\cref{algo:comp-coeff} to compute the local Taylor expansion coefficients of
$\vx(a)$ and $\lambda(a)$ near $(\vxk,\,\lambda_k)$ so that $\vH(\vx(a),\,
\lambda(a)) \approx \vz$. SANM then estimates the range of validity $a_r$ using
the standard ANM formulation in \eqnref{rov-taylor}. SANM also computes a
Pad{\'e} approximation from the Taylor coefficients and estimates its range of
validity $a_p$ using techniques outlined in \cref{sec:pade}. SANM chooses the
approximation with a larger range of validity and computes the next
approximation to extend the solution path until $\lambda(a)$ reaches
$\lambda_t$. \cref{algo:continue} summarizes this process.

% f}}}

\subsection{Reducing Error When Solving The Equational Form}
% f{{{
\label{sec:continue-with-eqn}

When solving a nonlinear system in the original ANM formulation $\vf(\vx)+\vv =
\vz$, SANM adopts a novel continuation method to reduce accumulated numerical
and approximation error. The principle is to modify $\vv$ at each iteration to
incorporate the residual $\vf(\vxk) + \lambda_k\vv - \vz$.  Specifically,
instead of using a fixed definition $\vH(\vx,\,\lambda) = f(\vx)+\lambda\vv$, we
change the definition of $\vH$ in \cref{algo:continue} at the beginning of each
iteration:
\begin{align}
    \V{H_k}(\vx,\,\lambda) = \vf(\vx)
        +(\lambda - \lambda_k)\vv
        -(1 - (\lambda - \lambda_k))\vf(\vxk)
\end{align}

Note that the loop invariant $\V{H_k}(\vxk,\,\lambda_k)=\vz$ is still satisfied.
We use $\lambda - \lambda_k$ to represent the progress made in the current
iteration. The final solution is found if $\lambda - \lambda_k=1$. Let
$\V{r_k} \defeq \vf(\vxk) + \vv$ denote the residual. In the continuation we
solve $\V{H_k}(\V{x_{k+1}},\,\lambda_{k+1})\approx\vz$, which implies
\begin{align}
    \V{r_{k+1}} \approx (1 - (\lambda_{k+1} - \lambda_k))\V{r_k}
\end{align}

Therefore, the continuation decreases the residual. Updating $\V{H_k}$
at each iteration automatically accounts for numerical error and approximation
error. We also change the loop condition to $\norm{\V{r_k}} \ge \epsilon$ to
achieve the desired solution residual $\epsilon$. This formulation allows us to
obtain very accurate solutions such as the low residual RMS shown in
\cref{tab:expr-gravity} and \cref{tab:expr-deform}.

% f}}}

% vim: tw=80 filetype=tex foldmethod=marker foldmarker=f{{{,f}}} spell spelllang=en

\section{Computing the Affine Transformations of Taylor Coefficients}
\label{sec:single-opr}

We present algorithms to determine the affine transformations of the
highest-order Taylor coefficient for a few nonlinear operators that are commonly
used in graphics applications. As we have shown in \cref{sec:coeff-prop}, these
operators can be combined into a computing graph to define the nonlinearity in
the target application. Although \cref{prop:linear} ensures the existence of the
linearity being sought in the general case, for certain operators, we can
compute the affine transformations directly without resorting to computing
higher-order derivatives. In this section, we use $f$ to represent the nonlinear
operator under investigation and also the value of this operator with $x$
assumed to be the independent variable. For any variable $x$, we use $\{x_i\}$
to denote its Taylor coefficients with respect to the path parameter~$a$.

Our goal is to derive an affine relationship between $f_k$ and $x_k$, assuming
that $x_0,\,\ldots,\, x_{k-1}$ and $f_0,\, \ldots,\, f_{k-1}$ are known
constants.

\subsection{Basic Arithmetic Operations}
% f{{{
The affine transformations for the four basic arithmetic operations are derived
by equating the coefficient of $a^k$ in the expansion of the equation:
\begin{itemize}
    \item $f=x+y$: By introducing $f=\sum_{k=0}^Nf_ka^k$, $x=\sum_{k=0}^Nx_ka^k$,
        and $y=\sum_{k=0}^Ny_ka^k$ and equating the coefficient of $a^k$, we
        have $f_k=x_k+y_k$.
    \item $f=x-y$: Similarly, we have $f_k=x_k-y_k$.
    \item $f=xy$: Similarly, we have $f_k=y_0x_k + x_0y_k +
        \sum_{i=1}^{k-1}x_iy_{k-i}$.
    \item $f=\nicefrac{x}{y}$: We have $x=fy$, which implies $x_k=f_0y_k +
        y_0f_k + \sum_{i=1}^{k-1}f_iy_{k-i}$, and therefore
        $f_k = \qty(\nicefrac{1}{y_0})x_k - \qty(\nicefrac{f_0}{y_0})y_k -
        \nicefrac{1}{y_0}\sum_{i=1}^{k-1}f_iy_{k-i}$.
\end{itemize}
% f}}}

\subsection{Elementwise Analytic Functions}
% f{{{
A function $\V{y}=\vf(\V{x})$ is said to be elementwise if $\forall i: y_i =
f(x_i)$. For such functions, we only need to derive the affine transformation
for the univariate case.

The problem of computing the Taylor coefficients of $f(\sum_{i=0}^k x_i a^i)$
given the coefficients $\{x_i\}$ and the Taylor expansion of $f(x_0+a)$ is known
in the literature as the composition problem, for which there exists a fast
$O((k\log k)^{3/2})$ algorithm \citep{ brent1978fast}. However, it is not
necessarily fast when $k$ is small, and the implementation is complicated.

As shown in \citet{griewank2008evaluating}, for most functions of practical
interest, we can find auxiliary functions $a(x)$, $b(x)$, and $c(x)$ whose
Taylor coefficients $\{a_i\}$, $\{b_i\}$, and $\{c_i\}$ can be easily computed
given $\{x_i\}$, such that
\begin{align}
    b(x)f'(x) - a(x)f(x) = c(x)
\end{align}
The coefficients $\{f_i\}$ are then computable in $O(k^2)$ time via a formula
involving $\{a_i\}$, $\{b_i\}$, $\{c_i\}$, and $\{x_i\}$. \citet{
griewank2008evaluating} provide a thorough treatment on this subject. We list in
\cref{tab:func-coeff-prop} the result formulas to propagate Taylor coefficients
through the elementwise analytic functions currently used in SANM for mesh
deformation applications:

\begin{table}[h]
    \small
    \centering
    \caption{Expansion coefficient propagation for univariate
        functions~\citep{griewank2008evaluating}}
    \label{tab:func-coeff-prop}
    \begin{tabular}{ll}
        \toprule
        $f(x)$ & Recurrence for $f_k$ \\
        \midrule
        $\ln(x)$ & $\frac{1}{x_0}\qty(
            x_k - \sum_{i=1}^{k-1}\frac{i}{k}x_{k-i}f_i)$ \\[1.5ex]
        $x^r$ & $\frac{1}{x_0}\qty(rf_0x_k +
            \sum_{i=1}^{k-1}(\frac{i}{k}(r+1) - 1)f_{k-i}x_i)$ \\
        \bottomrule
    \end{tabular}
\end{table}

Note that when $f(x)=x^r$ and $r$ is an integer, the recurrence is numerically
unstable when $|x|$ is small. In this case, we compute the Taylor coefficients
via exponentiation by squaring for polynomials in $O(k^2 \log r)$ time.
% f}}}

\subsection{Matrix Inverse}
% f{{{
Let $\vF = \vX^{-1}$ where $\vX\in\real^{m\times m}$ is a square matrix.
Introduce the power series definition and rearrange the terms:
\begin{align}
    \qty(\sum_{i=0}^k \V{F_i}a^i)\qty(\sum_{i=0}^k \V{X_i}a^i) = \V{I}
\end{align}
The coefficient of $a^k$ on the left hand side is $\sum_{i=0}^k \V{F_i}
\V{X_{k-i}}$, which must be zero because the right hand side is a constant:
\begin{align}
    & \sum_{i=0}^k \V{F_i}\V{X_{k-i}} = \vz \nonumber \\
    \label{eqn:taylor-mat-inv}
    \implies & \V{F_k} = -\V{X_0}^{-1}\V{X_k}\V{X_0}^{-1} -
        \qty(\sum_{i=1}^{k-1}\V{F_i}\V{X_{k-i}})\V{X_0}^{-1}
\end{align}
Equation \eqnref{taylor-mat-inv} explicitly defines an affine relationship
between $\V{F_k}$ and $\V{X_k}$.
% f}}}

\subsection{Matrix Determinant}
% f{{{
Let $g(a) = f(\V{X}(a))= \det(\sum_{i=0}^k \V{X_i} a^i)$ where
$\V{X}\in\real^{m\times m}$. A straightforward method to compute $g_k$ is to
expand the determinant according to the Leibniz formula and compute polynomial
products, which incurs exponential complexity $O(m!k^2)$ in terms of the matrix
size. Although this suffices for FEM applications in 2D or 3D (with $m=2$ or
$m=3$), we also present a method with polynomial complexity that is better
suited for larger matrices.

The terms containing $X_{kij}a^k$ that contribute to $g_k$ can only be
multiplied with $X_{0i'j'}$ where $i'\neq i$ and $j' \neq j$. The multiplier of
$X_{kij}a^k$ is in fact $C_{ij}$, where $\V{C}$ is the cofactor of $\V{X_0}$
with $C_{ij}$ defined as the determinant of the remaining matrix by removing the
$\nth{i}$ row and $\nth{j}$ column of $\V{X_0}$. Therefore:
\begin{align}
    g_k = \trans{\vect(\V{C})} \vect(\V{X_k})+ q_k
\end{align}

\paragraph{Computing the slope:} To efficiently compute the cofactor matrix~$
\V{C}$, we use the identity of Cramer's rule:
\begin{align}
    \V{X_0}^{-1} &= \frac{1}{\det(\V{X_0})}\transv{C}
\end{align}
Computing the matrix inverse incurs numerical stability issues. Instead, we
first compute the SVD decomposition
$\V{X_0} = \V{U}\V{\Sigma}\transv{V}$ and then compute $\V{C}$ as:
\begin{align}
    \V{C} &= \det(\V{X_0})\invtransv{X_0} \nonumber \\
    &= \det(\V{U})\cdot\det(\V{\Sigma})\cdot\det(\V{V})\cdot
        \V{U}\V{\Sigma}^{-1}\transv{V} \nonumber \\
    &= \det(\V{U})\cdot\det(\V{V})\cdot
        \V{U}\V{D}\transv{V}
\end{align}
Here $\V{D}=\det(\V{\Sigma})\V{\Sigma}^{-1}$ is a diagonal matrix and $D_{ii} =
\prod_{j\neq i}\Sigma_{jj}$. Such a formulation avoids division of singular
values and is stable even for ill-conditioned matrices.

\paragraph{Computing the bias:}
Similar to the argument in \cref{prop:linear}, we have $q_k =
\coeffof{k}{\det(\sum_{i=0}^{k-1}\V{X_i}a^i)}$. This is known as the polynomial
matrix determinant problem. We propose an efficient solution using discrete
Fourier transform, which has also been discovered by \citet{hromvcik1999new}:
\begin{enumerate}
    \item Compute $\V{Y_i} = \sum_{l=0}^{k-1}\V{X_l}\omega_{K}^{il}$ for $0 \le
        i < K$ with Fast Fourier Transform (FFT), where $K$ is the next power of
        two after $k$, and $\omega_K = e^{-j\frac{1}{2\pi K}}$ is a $\nth{K}$
        root of unity. This step costs $O(k\log k m^2)$.
    \item Compute the determinants $d_i = \det(\V{Y_i})$ for $0 \le i < K$ in
        $O(km^3)$ time.
    \item Use the inverse discrete Fourier transform to compute $q_k =
        \frac{1}{K}\sum_{i=0}^{K-1}d_i \omega_K^{-ik}$ in $O(k)$ time.
\end{enumerate}
The above method computes the bias $q_k$ in $O((\log k + m)km^2)$ time.
% f}}}

\subsection{Singular Value Decomposition}
% f{{{

Singular Value Decomposition (SVD) generates three matrices from a single matrix
input: $\vU\vSig\trans{\vV} = \vX$ where $\vU$ and $\vV$ are orthonormal
matrices and $\vSig$ is a diagonal matrix containing the singular values in
decreasing order. Here we only consider the square case $\vX\in\real^{m\times
m}$. We use $\sigma_i = \Sigma_{ii}$ to represent the singular values.

Although the constitutive models considered in this paper do not directly use
SVD, the ARAP energy, defined as $\Psi_{\text{ARAP}}(\vF) \defeq \frac{\mu}{2}
\norm{\vF - \vR}$, involves a Polar Decomposition (PD) $\vF = \vP\vR$. PD can be
computed from SVD via $\vP = \vU\vSig\trans{\vU}$ and $\vR=\vU\trans{\vV}$. SVD
is also potentially useful for other applications. Therefore, we first present
how to compute SVD in the SANM framework and then discuss extra modifications to
compute PD more stably. The ARAP energy actually needs a \emph{rotation-variant}
PD to prevent reflections in $\vR$ by requiring that $\det(\vR)=1$. We also
discuss how to compute such rotation-variant SVD in SANM.

An obstacle in numerical differentiation of SVD is that when there are identical
singular values $\sigma_i = \sigma_j$, the corresponding singular vectors
$\matcol{\vU}{i}$, $\matcol{\vU}{j}$, $\matcol{\vV}{i}$, and $\matcol{\vV}{j}$
are not uniquely determined. The Jacobians $\pdv{\vU}{\vX}$ and $\pdv{\vV}{\vX}$
in this case are thus undefined since different perturbations on $\vX$ induce
noncontinuous changes in $\vU$ and $\vV$. This case corresponds to a division by
zero in the Jacobian computation, which is often circumvented by various
numerical tricks \citep{ papadopoulo2000estimating, liao2019differentiable,
seeger2017auto}.

In graphics applications, however, identical singular values occur frequently.
For example, the singular values of the deformation gradient matrix in isotropic
stretching are all identical. As a remedy, we propose to use an alternative form
of SVD that includes $\vU\trans{\vV}$ directly, which we denote by SVD-W:
\begin{align}
    \begin{array}{l}
        \vU\vSig\trans{\vU}\vW = \vX \\
        \vW = \vU\trans{\vV}
    \end{array}
\end{align}
Note that $\vW$ is also the rotation matrix in the polar decomposition of $\vX$,
which is unique when $\vX$ is invertible and the Jacobian $\pdv{\vW}{\vX}$ is
thus well-defined.

Now we present the derivation of affine transformations from $\vXk$ to $\vUk$,
$\vSigk$, and $\vWk$. We are not going to give a final equation because it will
be too complex and repeat most of the derivation. Instead, we focus on
explaining the overall procedure for deriving these affine transformations.

We start by expanding the product $\vU\vSig\trans{\vU}\vW$ and extracting the
coefficient of $a^k$, which should be equal to $\vXk$:
\begin{align}
    \label{eqn:svd-xk-e}
    \vXk &= \sum_{\substack{a + b + c + d = k \\ \min(a,b,c,d)\ge 1}}
        \V{U_a}\V{\Sigma_b}\trans{\V{U_c}}\V{W_d}  + \vE \\
    \label{eqn:svd-e-eq}
    \vE &= \vUk\vSigz\trans{\vUz}\vWz
        + \vUz\vSigk\trans{\vUz}\vWz \nonumber \\
        & \hspace{1em} + \vUz\vSigz\trans{\vUk}\vWz
             + \vUz\vSigz\trans{\vUz}\vWk
\end{align}
We define $\vF=\trans{\vUz}\vE\trans{\vWz}\vUz$. Introduce \eqnref{svd-e-eq}
to the right hand side:
\begin{align}
    \label{eqn:svd-f}
    \vF = \trans{\vUz}\vUk\vSigz + \vSigk + \vSigz\trans{\vUk}\vUz +
    \vSigz\trans{\vUz}\vWk\trans{\vWz}\vUz
\end{align}
Note that \eqnref{svd-xk-e} establishes an affine relationship between $\vXk$
and $\vE$, and $\vF=\trans{\vUz}\vE\trans{\vWz}\vUz$ is also linear. We now only
need to seek the affine transformations from $\vF$ to $\vUk$, $\vSigk$ and
$\vWk$.

Expand the orthogonality constraints $\trans{\vU}\vU=\vI$ and
$\trans{\vW}\vW=\vI$:
\begin{align}
    \label{eqn:svd-uktuz}
    \trans{\vUz}\vUk + \trans{\vUk}\vUz + \vBu &= \vz \\
    \label{eqn:svd-wktwz}
    \trans{\vWz}\vWk + \trans{\vWk}\vWz + \vBw &= \vz \\
    \text{where } \vBu = \sum_{i=1}^{k-1} \trans{\V{U_i}}\V{U_{k-i}}
    \text{ and } & \vBw = \sum_{i=1}^{k-1} \trans{\V{W_i}}\V{W_{k-i}} \nonumber
\end{align}

\paragraph{Solving $\vSigk$:} From the constraints \eqnref{svd-uktuz} and
\eqnref{svd-wktwz}, we have
\begin{align}
    \label{eqn:svd-u-diag}
    \diag(\trans{\vUz}\vUk) &= -\frac{1}{2}\diag(\vBu) \\
    \label{eqn:svd-w-diag}
    \diag(\trans{\vUz}\vWk\trans{\vWz}\vUz) &=
        -\frac{1}{2}\diag(\trans{\vUz}\vBw\vUz)
\end{align}
Introducing \eqnref{svd-u-diag} and \eqnref{svd-w-diag} into \eqnref{svd-f}
allows us to solve $\diag(\vSigk)$ from $\diag(\vF)$.

\paragraph{Solving $\vWk$:} Because $\trans{\vUz}\vUk\vSigz + \vSigk +
\vSigz\trans{\vUk}\vUz$ is symmetric, we can cancel this term in \eqnref{svd-f}
by subtracting $\trans{\vF}$ from $\vF$:
\begin{align}
    \label{eqn:svd-f-ft}
    \vF - \trans{\vF} = \vSigz\trans{\vUz}\vWk\trans{\vWz}\vUz -
    \trans{\vUz}\vWz\trans{\vWk}\vUz\vSigz
\end{align}

From \eqnref{svd-wktwz} we have $\trans{\vWk} = -\vBw\trans{\vWz} - \trans{\vWz}
\vWk \trans{\vWz}$. Introduce it to~\eqnref{svd-f-ft}:
\begin{align}
    \label{eqn:svd-f-ft-const}
    \vF - \trans{\vF} - \trans{\vUz}\vWz\vBw\trans{\vWz}\vUz\vSigz &=
        \vSigz\trans{\vUz}\vWk\trans{\vWz}\vUz \nonumber \\
        & \hspace{2em} + \trans{\vUz}\vWk\trans{\vWz}\vUz\vSigz
\end{align}
Note that \eqnref{svd-f-ft-const} is a Sylvester equation in the form $\vSigz\vM
+ \vM\vSigz = \vA$ with $\vM = \trans{\vUz}\vWk\trans{\vWz}\vUz$. The solution
is $M_{ij} = \frac{A_{ij}}{\sigma_i + \sigma_j}$ and $\vWk = \vUz\vM
\trans{\vUz}\vWz$.

\paragraph{Solving $\vUk$:} Introducing $\trans{\vUk} = -\vBu\trans{\vUz} - \trans{\vUz}
\vUk \trans{\vUz}$ (derived from \eqnref{svd-uktuz}) and the solutions of $\vWk$
and $\vSigk$ into \eqnref{svd-f} results in another Sylvester equation: $\vSigz
\trans{\vUz}\vUk - \trans{\vUz}\vUk\vSigz = \V{B}$, with the solution
$(\trans{U_0}U_k)_{ij}=\frac{B_{ij}}{\sigma_i - \sigma_j}$.

\subsubsection{Polar Decomposition Case}
If other operators in the computing graph only need the $\vW$ output matrix of
the SVD-W operator, this computation can be further simplified by considering
the polar decomposition: $\vX=\vP\vW$.

Note that the symmetry of $\vP$ implies the symmetry of $\V{P_i}$. Expand the
identity $\trans{\vP}\vP = \vX\trans{\vX}$ and extract the coefficient of $a^k$:
\begin{align}
    \label{eqn:svd-polar}
    \vPk \vPz + \vPz \vPk + \vBp &=  \vXz\trans{\vXk} + \vXk\trans{\vXz}  \\
    \vBp &= \sum_{i=1}^{k-1} (\V{P_{i}}\V{P_{k-i}} -
        \V{X_{i}}\trans{\V{X_{k-1}}})
\end{align}

\paragraph{Solving $\vPk$:} Introduce $\vPz = \vUz\vSigz\trans{\vUz}$ to
\eqnref{svd-polar} and multiply both sides with $\trans{\vUz}$ and $\vUz$:
\begin{align}
    \trans{\vUz}\vPk\vUz\vSigz + \vSigz\trans{\vUz}\vPk\vUz =
    \trans{\vUz}(\vXz\trans{\vXk} + \vXk\trans{\vXz} - \vBp)\vUz
\end{align}
We have obtained another Sylvester equation and $\vPk$ can be solved similarly
to solving $\vWk$ from \eqnref{svd-f-ft-const}.

\paragraph{Solving $\vWk$:} We can derive $\vWk$ by directly expanding
$\vP\vW=\vX$:
\begin{align}
    \vWk &= \vUz\vSigz^{-1}\trans{\vUz}\qty(
        \vXk - \sum_{i=1}^k\V{P_i}\V{W_{k-i}}
    )
\end{align}

\subsubsection{Implementation Notes}

Although the SVD-W formulation provides a numerically stable expression to
compute the Jacobian $\pdv{\vW}{\vX}$, the biases~$\vWk$ still depend on the
numerically unstable $\V{U_i}$ via \eqnref{svd-xk-e}. The polar decomposition
formulation does not suffer from this problem because $\vP=\vU\vSig\trans{\vU}$
is unique in the presence of equal singular values as long as $\vX$ is
invertible. In the implementation, we transparently switch to the polar
decomposition formulation to compute $\vWk$ in the SVD-W operator when the
outputs $\vU$ and $\vSig$ are not needed by other operators in the computing
graph. We adopt the Lorentzian broadening \citep{liao2019differentiable} $
\nicefrac{x}{y} \rightarrow \nicefrac{xy}{(y^2+\epsilon)}$ with $\epsilon =
10^{-12}$ when computing the divisions in solving the Sylvester equations.

The rotation-variant SVD requires $\det(\vW)=1$ so that $\vW$ is a proper
rotation matrix. It is traditionally obtained by negating the last singular
value and the corresponding left-singular or right-singular vector if $
\det(\vW)=-1$ \citep{ kim2020dynamic}. However, when the last singular value is
identical to another singular value, the Jacobian $\pdv{\vW}{\vX}$ becomes
undefined because there are multiple singular vectors for this singular value,
and it is arbitrary to negate one of them. From a numerical perspective, we need
to compute $ \nicefrac{1}{(\sigma_i + \sigma_j)}$ in the Jacobian with some
$\sigma_i+ \sigma_j=0$. To improve numerical stability, we modify the
rotation-variant SVD computation by grouping identical singular values together
and preferring to negate all singular values and vectors in a group of an odd
size. In the 3D case, there is an odd number of singular values, and therefore
an odd-sized group must exist.

% f}}}

% vim: tw=80 filetype=tex foldmethod=marker foldmarker=f{{{,f}}} spell spelllang=en

\section{The SANM System}

\begin{listing*}[t]
\begin{minted}{cpp}
SymbolVar pk1(EnergyModel energy_model, const MaterialProperty& material, SymbolVar F) { switch (energy_model) {
    case EnergyModel::NEOHOOKEAN_I: {
        fp_t k = material.bulk_modulus(), mu = material.shear_modulus();
        SymbolVar FTinv = batched_mat_inv_mul(F, {}, true).batched_transpose(), J = F.batched_det(), Ic = F.pow(2).reduce_sum(-1),
                  J23 = J.pow(-2. / 3.), t2 = linear_combine({{mu / (-3._fp), J23 * Ic}, {k, J * J}, {-k, J}}, 0) * FTinv,
                  P = linear_combine({{mu, J23 * F}, {1._fp, t2}});
        return P;
    }
    case EnergyModel::NEOHOOKEAN_C: {
        fp_t mu = material.shear_modulus(), lambda = material.lame_first();
        SymbolVar FTinv = batched_mat_inv_mul(F, {}, true).batched_transpose(), J = F.batched_det(),
                  P = linear_combine({{mu, F}, {-mu, FTinv}, {lambda, J.log() * FTinv}});
        return P;
    }
    case EnergyModel::ARAP: {
        fp_t mu = material.shear_modulus();
        return (F - F.batched_svd_w(true)[2]) * mu;
    }
} }
\end{minted}
    \caption{A code excerpt for building the first Piola–Kirchhoff stress tensors
        with three constitutive models using the SANM framework. The symbolic
        nature of SANM allows easy substitution of different constitutive
        models. This function is essentially a literal translation of the
        formulas \eqnref{pk1-nc}-\eqnref{pk1-arap}.}
    \label{code:build}
\end{listing*}

We use C++ to implement SANM. This section discusses a few design choices that
support efficient mesh deformation applications in SANM.

\subsection{API Design}
\label{sec:API}

SANM adopts a define-and-run paradigm. The user describes a nonlinear system
symbolically and provides initial values and input/output transformations. SANM
automatically solves the system using the extended ANM framework described in
previous sections.

One of the most important public APIs of SANM is for building the symbolic
computing graph to represent the nonlinear system of interest. We adopt an
object-oriented design to enable intuitive and efficient computing graph
building. We use objects in the program to represent variable vertices in the
computing graph. We also provide overloading for common arithmetic operators.
\cref{code:build} shows a SANM code excerpt of expressing the first
Piola-Kirchhoff tensors for a few constitutive models. SANM only requires the
user to provide the C++ object that represents the whole nonlinear system, the
sparse affine transformations on the inputs and outputs (see \cref{sec:batch}),
and the initial values. The user does not need to modify the solver to work on
different tasks.

SANM significantly reduces programming effort for applying ANM. We roughly
measure programming effort by the number of lines of C++ code.  With SANM, the
whole FEM solver for all elastic deformation applications in this paper,
including auxiliary functionalities such as mesh input/output and tetrahedron
processing, needs only 1.5K lines of code without using external geometry
manipulation libraries. The SANM library itself has about 7.5K lines of code. By
contrast, the official ANM implementation of \citet{ chen2014asymptotic} has
11,499 lines of code for only the ANM numerical solving part
(\verb|neoHookeanANM.cpp| and \verb|neoHookeanANMForward.cpp|). Their code only
supports the incompressible neo-Hookean model, while SANM allows working with
multiple constitutive models by changing a few lines of code.

Internally, SANM provides a mechanism to register new operators. Each operator
only needs to implement a few functions, such as gradient computing and Taylor
coefficient propagation. The SANM framework manipulates the computing graph and
orchestrates individual operator functionalities to implement ANM solving. SANM
currently has operators to support the constitutive models used in this paper,
including arithmetic operators (addition, subtraction, multiplication, division,
logarithm, and exponentiation), tensor operators (slicing and concatenation),
and batched linear algebra operators (matrix multiplication, inverse,
determinant, transpose, and SVD). It is easy to extend SANM to support more
operators via the operator registration mechanism.

\subsection{Batch Computing}
\label{sec:batch}

In finite element analysis, we typically carry out computation of some identical
form for a set of elements. In the context of 3D mesh deformation, we compute
the stress tensors of each tetrahedron from its deformation gradient. This form
of computation allows us to exploit modern computer architectures better via
batch processing. Specifically, we pack the matrices of all elements into a
large tensor on which the operator in the computing graph executes. For example,
the deformation gradient of the $\nth{i}$ element is computed as $\vFi =
\operatorname{matmul}\qty(\vD_{si},\, \vD_{mi}^{-1})$. If we execute the
coefficient propagation algorithm discussed in \cref{sec:coeff-prop} for each
tetrahedral element independently, we will need to invoke the \verb|matmul|
operators $n$ times, while each invocation only works on $3\times 3$ small
matrices. Instead, we pack all the deformation gradients into a third-order
$n\times3\times3$ tensor $\mathbb{F} = [\V{F_1};\, \cdots;\,\V{F_n}]$, and pack
$\mathbb{D}_m$ and $\mathbb{D}_s$ similarly. The computing graph then contains a
single \verb|batched_matmul| operator that computes matrix multiplication for
all elements together: $\mathbb{F} = \operatorname{batched\_matmul}
\qty(\mathbb{D}_s,\, \mathbb{D}_m^{-1})$. In our mesh deformation applications,
the unknown vector $\vx$ represents node coordinates. We use a sparse affine
transformation $\vA$ to map the coordinates to shape matrices and another
transformation $\vB$ to map from stress tensors to nodal force: $\vf(\vx) =
\vB(\vP(\vF(\vA(\vx))))$ where $\vP$ represents the Piola–Kirchhoff stress
tensors, $\vF$ represents the deformation gradients, and $\vP(\vF(\cdot))$ is
computed in a batched manner. The sparse affine transformations $\vA$ and $\vB$
are provided as input/output transformations to the SANM solver.

Batch computing improves performance by better utilizing the hardware, although
it does not reduce computational complexity. Most modern CPUs support Single
Instruction Multiple Data (SIMD) parallelism, and GPUs are designed to process
large amounts of data in parallel. Without batch computing, such hardware
capability can hardly be utilized by the small matrices occurring in finite
element analysis. Batch computing also amortizes the overhead of computing graph
traversing in \cref{algo:comp-coeff}.

SANM supports parallel computing by splitting data on the batch dimension, which
is managed by the framework and is oblivious to individual operator
implementations. \cref{fig:scalability} compares solving times achieved with
different numbers of threads, which shows that SANM exhibits modest scalability
on practical workloads.

\begin{figure}[ht]
    \centering
    \includegraphics[width=0.9\columnwidth]{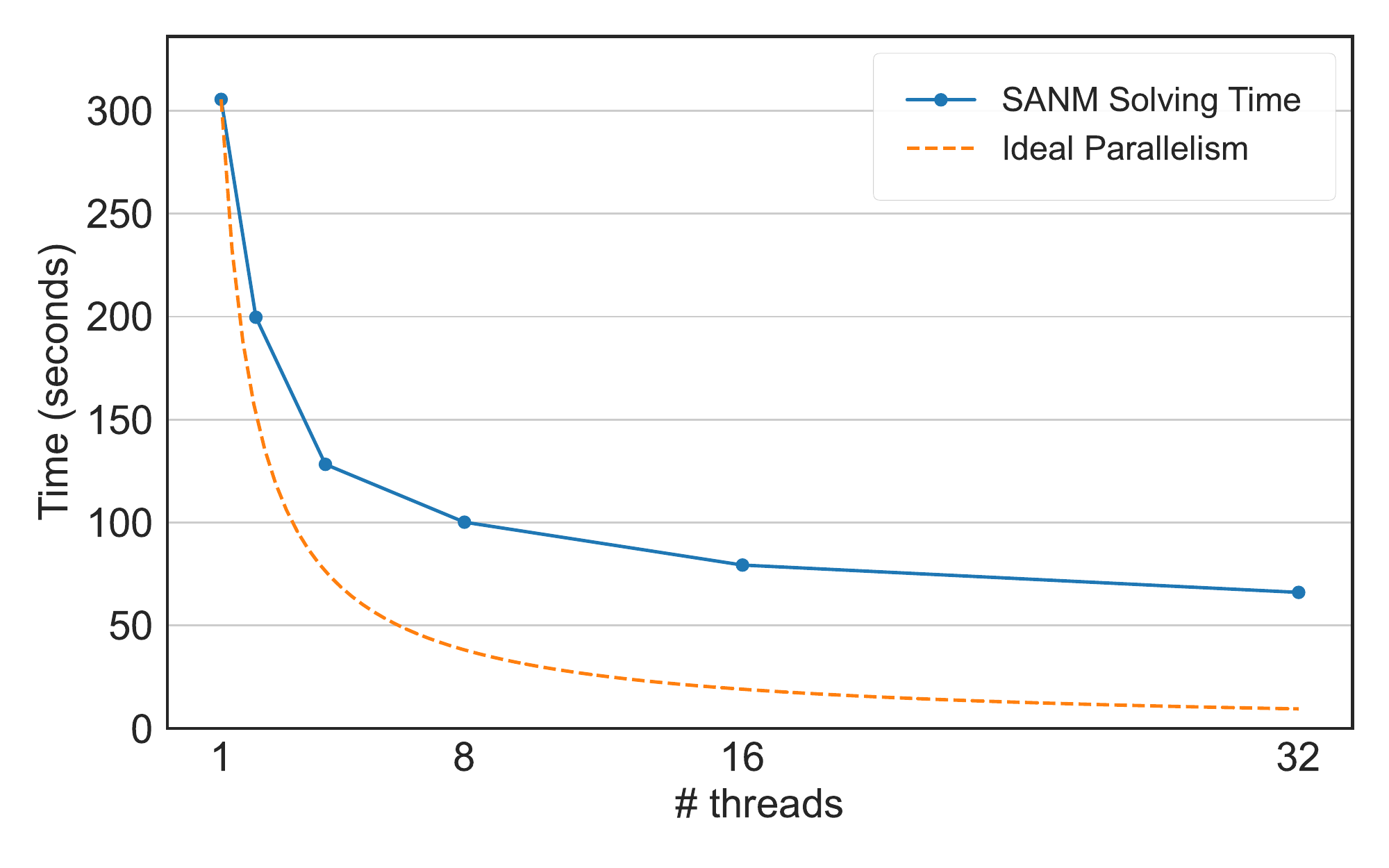}
    \caption{Comparing the end-to-end solving times with different numbers of
        threads on the static forward gravity equilibrium problem of the
        Armadillo model with 221,414 nodes and 696,975 tetrahedrons. SANM
        exhibits modest scalability using up to 32 threads. The ideal
        parallelism is the single-thread solving time divided by the number of
        threads.
    }
    \label{fig:scalability}
\end{figure}

\subsection{Performance Optimizations}

We design SANM with high-performance computing in mind. Here we discuss other
optimizations besides batch computing.

\paragraph{Efficient numerical computing primitives:} SANM provides an
abstraction of numerical computing primitives. The implementations of computing
graph operators invoke these primitives instead of directly working on numerical
data.  This paradigm allows us to separate numerical algorithm description from
performance engineering. Currently, we use Eigen~\citep{eigenweb} and Intel Math
Kernel Library to implement the computing primitives on CPU with Single
Instruction Multiple Data (SIMD) optimizations. We can easily extend SANM to
support GPU by adding another GPU backend for the primitives without modifying
implementations of operators or the ANM solver. This abstraction also allows
SANM to benefit from other research on optimizing tensor computing performance,
such as recent related research in deep learning \citep{ chen2018tvm,
jouppi2017in}.

\paragraph{Automatic memory management with copy-on-write:} SANM automatically
manages tensor memory by reference counting, with eager memory sharing and
copy-on-write to simplify programming without sacrificing performance. When a
tensor object is copied, only a new reference is stored in the destination. When
a tensor object is modified, SANM makes a private copy before modification if
the reference count is greater than one.

\paragraph{Sparse affine transformations:} We exploit the structural sparsity
when computing the Jacobians and the affine transformations in
\cref{algo:comp-coeff}. For a batch-packed tensor with dimensions $n\times
m\times m$, we use an $n\times m^2\times m^2$ tensor to represent its Jacobian
instead of using a full $nm^2 \times nm^2$ matrix because each matrix in the
batch is independent of each other. Our sparse Jacobian representation
significantly reduces memory usage when $n$ is large. It is also friendly to
batch computing.  Furthermore, we use a compact $n \times m^2$ matrix to
represent Jacobians for elementwise operators.

\paragraph{Special handling of zero tensors:} Zero tensors frequently occur,
such as being used as the initial accumulation value. SANM retains a special
buffer shared by all zero-initialized tensors (also with reference counting and
copy-on-write). Thus checking whether a tensor is all zero can be easily done by
comparing the buffer address. This design allows implementing a fast path for
handling zero inputs in elementary arithmetic operators, such as $x+0=x$ and
$x\cdot0=0$. This optimization leads to a 7.91\% speedup in our experiment.

\subsection{Future Performance Optimizations}

We discuss other optimizations that are not yet implemented but likely to be
helpful. Thanks to the define-and-run paradigm, SANM users can benefit from
future optimizations by simply updating their SANM library without modifying
their application code.

\paragraph{Computing graph optimization:} Since the user symbolically defines
the computing graph, SANM can optimize the graph before starting numerical
computation. For example, we can simplify arithmetic expressions. We can also
fuse arithmetic operators with just-in-time compilation. There is a large body
of research on traditional compiler optimization \citep{ lattner2004llvm} and
recent tensor compiler optimization in deep learning \citep{lattner2020mlir}
that may benefit future SANM optimizations.

\paragraph{Reducing memory usage:} Currently, we store all the intermediate
Taylor coefficients in memory, which incurs some memory overhead. A possible
improvement is setting up checkpoints on the computing graph and recomputing the
Taylor coefficients between checkpoints each time. A good choice of checkpoints
might induce little computational cost ($O(1)$ times the original cost) while
saving lots of memory ($O(\nicefrac{1}{\sqrt{N}})$ relative memory usage) for a
computing graph with a chain of length $N$ \citep{chen2016training}.

% vim: tw=80 filetype=tex foldmethod=marker foldmarker=f{{{,f}}} spell spelllang=en

\section{Mesh Deformation Applications}
\label{sec:application}

% f{{{
We evaluate SANM on a few volumetric mesh deformation problems. We first briefly
review the basics of elastic deformation analysis. Readers may refer to
\citet{bonet_wood_2008, sifakis2012fem, kim2020dynamic} for a more thorough
treatment on this subject.

We consider 3D deformation of hyperelastic materials, for which the work done by
the stresses during a deformation process only depends on the initial and final
state. A constitutive model relates the elastic potential energy density $\Psi$
to the deformation gradient $\vF$. Under a piece-wise linear tetrahedral
discretization, the deformation gradient is constant within a tetrahedron: $\vF
= \vD_s\vD_m^{-1} \in \real^{3\times3}$ where $\vD_s$ is the \emph{deformed
shape matrix} (computed from the deformed tetrahedron) and $\vD_m$ is the
\emph{reference shape matrix} (computed from the rest tetrahedron). A shape
matrix of a tetrahedron packs the three column vectors from one vertex to the
other three. The elastic force $\V{f_{ij}} \in \real^3$ exerted by a single
tetrahedron $i$ to its $\nth{j}$ node is then derived by taking the gradient of
the potential energy with respect to the node coordinates: $\V{f_{ij}} =
-\pdv{\Psi(\V{F_i})} {\V{x_{ij}}}$. We can obtain the formulation for the
internal elastic force at a node $i$ by combining forces exerted by neighboring
tetrahedrons, which equals the following:
\begin{align}
    \vfi &=  \sum_{t\in N_i} \vP(\V{F_t}) \bar{\V{n}}_{t,i}
\end{align}
where $N_i$ is the set of adjacent tetrahedrons containing node $\vxi$,
$\vP(\vF) \defeq \pdv{\Psi(\vF)}{\vF}$ is the first Piola–Kirchhoff stress
tensor of the constitutive model, and $\bar{\V{n}}_{t,i}$ is the outward
area-weighted normal vector at node $\vxi$ of the tetrahedron $t$ in the
undeformed state.

To solve the inverse deformation problem that seeks a rest shape which deforms
to a given shape, we introduce the Cauchy stress tensor $\vsig$ that linearly
relates the elastic force to the deformed state:
\begin{align}
    \vfi &=  \sum_{t\in N_i} \vsig(\V{F_t}) \V{n}_{t,i} \\
    \vsig(\vF) &= \frac{1}{\det(\vF)}\vP(\vF)\trans{\vF}
\end{align}
where $\V{n}_{t,i}$ is the outward area-weighted normal vector at node $\vxi$ of
the tetrahedron $t$ in the given deformed state.

This paper considers three constitutive models: the compressible neo-Hookean
energy (denoted by NC), the incompressible neo-Hookean energy (denoted by NI),
and the As-Rigid-As-Possible energy (denoted by ARAP). We list their first
Piola–Kirchhoff stress tensors:
\begin{align}
    \label{eqn:pk1-nc}
    \vP_{NC}(\vF) &= \mu\qty(\vF - \invtrans{\vF}) +
        \lambda \log(J)\invtrans{\vF} \\
    \label{eqn:pk1-ni}
    \vP_{NI}(\vF) &= \mu J^{-\frac{2}{3}}\qty(\vF - \frac{1}{3} \norm{\vF}
        \invtrans{\vF}) + \kappa J (J-1) \invtrans{\vF} \\
    \label{eqn:pk1-arap}
    \vP_{ARAP}(\vF) &= \mu\qty(\vF - \vR) \\
    J &\defeq \det(\vF) \nonumber
\end{align}
Note that ARAP \eqnref{pk1-arap} needs a rotation-variant polar decomposition
$\vF = \vR\V{S}$ such that $\vR$ is a proper rotation matrix with $\det(\vR)=1$.
The parameters, $\lambda$ (Lam{\'e}'s first parameter), $\mu$ (Lam{\'e}'s second
parameter), and $\kappa$ (the bulk modulus) are all determined by the physical
properties of the material.
% f}}}

\subsection{Forward and Inverse Static Equilibrium Problems}
% f{{{

\begin{table}[t]
    \centering
    \caption{Comparing with the hand-coded, specialized ANM solver of \citet{
        chen2014asymptotic} on gravity equilibrium problems. In this table, inv.
        and fwd. mean inverse and forward problems respectively, and mt4 or mt6
        indicate using 4 or 6 threads for parallel computing. Since the code of
        \citet{ chen2014asymptotic} no longer compiles on modern systems, we
        directly use the data reported in their paper. We use the models
        included in their open-source release and exclude models with additional
        external force because their force description file is in a custom
        binary format. We ran SANM on a server with two Intel Xeon Platinum
        8269CY CPUs (2.5GHz - 3.8GHz), while \citet{ chen2014asymptotic} used a
        desktop PC with an Intel i7-3770K CPU (3.5GHz - 3.9GHz). We use the same
        truncation order $N=20$ as \citet{ chen2014asymptotic}. The data show
        that SANM delivers comparable, if not better, performance as a
        hand-coded and manually optimized ANM solver.  \label{tab:expr-cmp-chen}
    }
    \begin{tabular}{lrrrrr}
    \toprule
    \multirow{2}{*}{Model} & \multicolumn{3}{c}{Ours: SANM} & \multicolumn{2}{c}{\citet{chen2014asymptotic}} \\
    \cmidrule(lr){2-4}
    \cmidrule(lr){5-6}
        & \#Iter. & Time & Time (mt4) & \#Iter. & Time (mt6) \\
    \midrule
    \expandableinput img/chen-cmp.tex
    \bottomrule
    \end{tabular}
\end{table}

Given a static external force, we consider the problems of solving the rest
shape given the deformed shape (the inverse problem) and solving the deformed
shape given the rest shape (the forward problem) similar to \citet{
chen2014asymptotic}. Formally, let $\bar{\vx}$ denote the rest shape, $\vx$ the
deformed shape, $\vf(\bar{\vx},\,\vx)$ the internal elastic force, and
$\vf_{ext}$ the external force. We solve $\vf([\bar{\vx};\;\vx_b],\,
[\vx;\;\vx_b])+\vf_{ext} = \vz$ given either $\bar{\vx}$ or $\vx$, where $\vx_b$
contains fixed boundary nodes. This problem naturally fits into the numerical
continuation framework by replacing $\vf_{ext}$ with $\lambda \vf_{ext}$. We
consider static equilibrium under gravity and set $\vf_{ext}$ as the per-node
gravity.

\citet{chen2014asymptotic} has shown that ANM is tens to thousands of times
faster than the Levenberg-Marquardt algorithm on the inverse problem, and it is
also roughly ten times faster than an implicit Newmark simulation with kinetic
damping \citep{umetani2011sensitive} on the forward problem.
\cref{tab:expr-cmp-chen} compares SANM with the hand-coded and manually
optimized ANM solver of \citet{chen2014asymptotic}, which shows that SANM
achieves comparable or better performance while automatically solves the
problem. We also evaluate SANM on a large Armadillo model with 221,414 nodes and
696,975 tetrahedrons, which is nearly ten times larger than the models used in
\citet{chen2014asymptotic}. \cref{fig:scalability} presents the solving time of
the forward problem using different numbers of threads. \cref{fig:armadillo-g}
shows the intermediate states for a few values of $\lambda$ in the continuation.

\subsubsection{Comparison with Newton's methods}

\label{sec:application:cmp-newton-fwd}

\begin{table*}[t]
\centering
\small
\begin{threeparttable}
    \caption{
        Performance comparison on forward gravity equilibrium problems. We limit
        the Levenberg-Marquardt algorithm to use no more than 1000 iterations.
        The SANM speedup is computed by comparing with the fastest correct
        (i.e., producing no inverted tetrahedrons) alternative method for each
        problem.  The NC and NI materials refer to compressible and
        incompressible neo-Hookean materials respectively. $\rms(\vf)$ refers to
        the root-mean-square value of force residuals on unconstrained nodes.
        The bracketed numbers indicate the Gauss-Newton refinement iterations
        (limited to be 20). Proj. Newton uses per-tetrahedron Hessian projection
        derived from \citet{smith2019analytic}. SANM achieves the low residual
        without extra refinement thanks to the techniques presented in
        \cref{sec:continue-with-eqn}. Bold times mark the fastest methods and
        italic times mark the second fastest methods. SANM achieves an average
        speedup of \gmeanSpeedupGravity~by geometric mean.
        \label{tab:expr-gravity}
    }
    \begin{tabular}{llrrrrrrrrrrrrr}
        \toprule
        \multirow{2}{*}{Mesh} &
        \multirow{2}{*}{Material} &
        \multicolumn{3}{c}{Ours: SANM} &
        \multicolumn{3}{c}{Newton} &
        \multicolumn{3}{c}{Proj. Newton} &
        \multicolumn{3}{c}{Levenberg-Marquardt} &
        \multirow{2}{4em}{SANM Speedup} \\
        \cmidrule(lr){3-5}
        \cmidrule(lr){6-8}
        \cmidrule(lr){9-11}
        \cmidrule(lr){12-14}
        & &
        \#Iter & Time & $\rms(\vf)$ &
        \#Iter & Time & $\rms(\vf)$ &
        \#Iter & Time & $\rms(\vf)$ &
        \#Iter & Time & $\rms(\vf)$ & \\
        \midrule
        \expandableinput img/gravity.tex
        \bottomrule
    \end{tabular}
    \begin{tablenotes}
        \item[*] The solution contains inverted tetrahedrons.
    \end{tablenotes}
\end{threeparttable}
\end{table*}

We compare with more methods on the forward problem. The minimum total potential
energy principle dictates that the equilibrium state is the minimizer of the
total potential energy, including the elastic potential energy and the
gravitational potential energy. Therefore, an alternative method for the forward
problem is to solve $\argmin_{\vx} \qty(\Psi(\vx) - \transv{g} \vx)$ where
$\V{g}$ is the per-node gravity. We implement Newton's method with backtracking
line search to solve the minimization. We also evaluate positive-semidefinite
Hessian projection with a state-of-the-art derivation of per-tetrahedron
analytic eigensystems for the elastic energy functions \citep{
smith2019analytic}. Note that the energy minimization method does not apply to
the inverse problem due to the lack of corresponding global energy. We also
compare with directly minimizing $\norm{\vf(\vx)+\V{g}}$ by the
Levenberg-Marquardt algorithm.

We run the experiments on a desktop PC with an AMD Ryzen Threadripper 2970WX
CPU. All the implementations are compiled with the same compiler, use the same
linear algebra libraries, and use a single thread. Newton's method uses an LU
solver, and the projective Newton uses a faster LLT solver due to the guaranteed
positive definiteness of projected Hessians. We use techniques described in
\cref{sec:continue-with-eqn} to reduce the error of SANM. We set the truncation
order $N=20$. We set the convergence threshold to be $\epsilon=10^{-10}$ for the
RMS of the force residual. We find that energy minimization with Newtonian
methods often fails to converge to such a small force residual due to vanishing
step sizes near the optimum, and therefore we stop them if either the RMS of the
force residual or the change of $\vx$ in one iteration drops below $10^{-6}$. We
then use additional Gauss-Newton iterations to fine-tune the solution.
\cref{tab:expr-gravity} presents the comparisons, which shows that SANM
converges faster than the considered alternative methods in most cases and
achieves low residual without additional refinement.

% f}}}

\subsection{Controlled Mesh Deformation}
% f{{{
\label{sec:controlled-deform}

\begin{table*}[t]
\centering
\begingroup
\small
\setlength{\tabcolsep}{3pt}
\begin{threeparttable}
    \caption{
        Performance comparison on controlled mesh deformation problems. This
        table uses similar notations as \cref{tab:expr-gravity}. Note that
        Newton's minimization methods do not work with neo-Hookean energies
        because the initial guess contains inverted tetrahedrons. SANM uses
        equation solving presented in \cref{sec:continue-with-eqn} to refine the
        solution, while Newton's methods use Gauss-Newton iterations for
        refinement. SANM achieves an average speedup of \gmeanSpeedupDeform~by
        geometric mean.
        \label{tab:expr-deform}
    }
    \begin{tabular}{lrrrrrrrrrr@{\hspace{1.5em}}rrrrrr}
        \toprule
        \multirow{2}{*}{Mesh} &
        \multicolumn{3}{c}{Ours: SANM (ARAP)} &
        \multicolumn{3}{c}{Newton (ARAP)} &
        \multicolumn{3}{c}{Proj. Newton (ARAP)} &
        \multirow{2}{3em}{SANM Speedup} &
        \multicolumn{3}{c}{SANM (NC)} &
        \multicolumn{3}{c}{SANM (NI)} \\
        \cmidrule(lr){2-4}
        \cmidrule(lr){5-7}
        \cmidrule(lr){8-10}
        \cmidrule(lr){12-14}
        \cmidrule(lr){15-17}
        &
        \#Iter & Time & $\rms(\vf)$ &
        \#Iter & Time & $\rms(\vf)$ &
        \#Iter & Time & $\rms(\vf)$ &
        &
        \#Iter & Time & $\rms(\vf)$ &
        \#Iter & Time & $\rms(\vf)$ \\
        \midrule
        \expandableinput img/deform.tex
        \bottomrule
    \end{tabular}
    \begin{tablenotes}
        \item[*] The solution contains inverted tetrahedrons.
    \end{tablenotes}
\end{threeparttable}
\endgroup
\end{table*}

We demonstrate implicit homotopy solving on controlled mesh deformation
problems. The problem asks for the equilibrium state when specific nodes are
moved to given locations. The constrained nodes, called the control handles, are
usually specified by a user so that the elastic body can be deformed into the
desired pose. This problem is typically solved with an energy minimization
framework that minimizes the total elastic potential energy.

We propose an alternative approach with implicit homotopy. Given the rest shape
of the body as $\vx_0$, the initial position of control handles as $\vx_c$, and
the target position of control handles as $\vx_t$, we define an implicit
homotopy for the unconstrained nodes $\vx$:
\begin{align}
    \vH(\vx,\,\lambda) &=  \vf\qty(
        \qty[
            \begin{array}{c}
                \vx_0 \\
                \vx_c
            \end{array}
        ],\,
        \qty[
            \begin{array}{c}
                \vx \\
                \vx_c + \lambda(\vx_t - \vx_c)
            \end{array}
        ]) = \vz \\
    \text{with } \vH(\vx_0,\,0) &= \vz \nonumber
\end{align}
where $\vf(\bar{\vx},\,\vx)$ computes the internal elastic force given the rest
shape $\bar{\vx}$ and the deformed shape $\vx$. Numerical continuation of
$\lambda$ from $0$ to $1$ solves the coordinates of unconstrained nodes, and
each intermediate configuration is a valid equilibrium state with the control
handles on a linear path from $\vx_c$ to $\vx_t$. Computing an intermediate
state at $\lambda_i$ only requires solving an equation $\lambda(a)=\lambda_i$
and then evaluate $\vx(a)$, which incurs negligible extra cost. After
solving the homotopy, we refine the solution by solving a static equilibrium
problem with zero external force using the improved equation solving presented
in \cref{sec:continue-with-eqn} (with truncation order $N=6$) to reduce the
force residual.

We compare SANM with Newton's methods on a few 3D models with manually specified
target positions of control handles, such as the deformed Bob shown in
\cref{fig:summary:c}. Experimental settings are similar to those in
\cref{sec:application:cmp-newton-fwd}. Newton's methods directly minimize
$\Psi([\vx;\; \vx_t])$ starting at $\Psi([\vx_0;\;\vx_t])$. We do not evaluate
the Levenberg-Marquardt algorithm since it is too inefficient compared to
others.  Note that the initial guess $\vx=\vx_0$ for Newton's methods induces
inverted tetrahedrons. Unfortunately, the neo-Hookean energies do not handle
this case (corresponding to $J<0$ in \eqnref{pk1-nc} and \eqnref{pk1-ni}), and
therefore we do not evaluate Newton's methods on them.  Furthermore, a
straightforward Newton's method without Hessian projection sometimes fails on
the ARAP energy because the inverted tetrahedrons cause indefinite Hessian with
more negative eigenvalues that obstruct the optimization progress. By contrast,
numerical continuation in SANM ensures a smooth process where no tetrahedron
gets inverted, and therefore it also works for neo-Hookean energies.
\cref{tab:expr-deform} presents the results, which shows that SANM is more
robust and more efficient in most cases.

We demonstrate the robustness of SANM on a problem with larger deformation. We
twist a horizontal bar by 360 degrees and then bend it as shown in
\cref{fig:cube-twist}. Note that the target boundary configuration is
indistinguishable from another one which has only the bending but no twisting.
SANM naturally handles this case by using a piecewise-linear description of the
movement path. However, to apply energy minimization methods, a similar but
arguably less principled continuation scheme (such as minimizing the energy in
multiple stages) or more complicated initialization strategies are needed to
resolve the rotation ambiguity.

\begin{figure}[ht]
    \centering
    \includegraphics[width=.3\columnwidth]{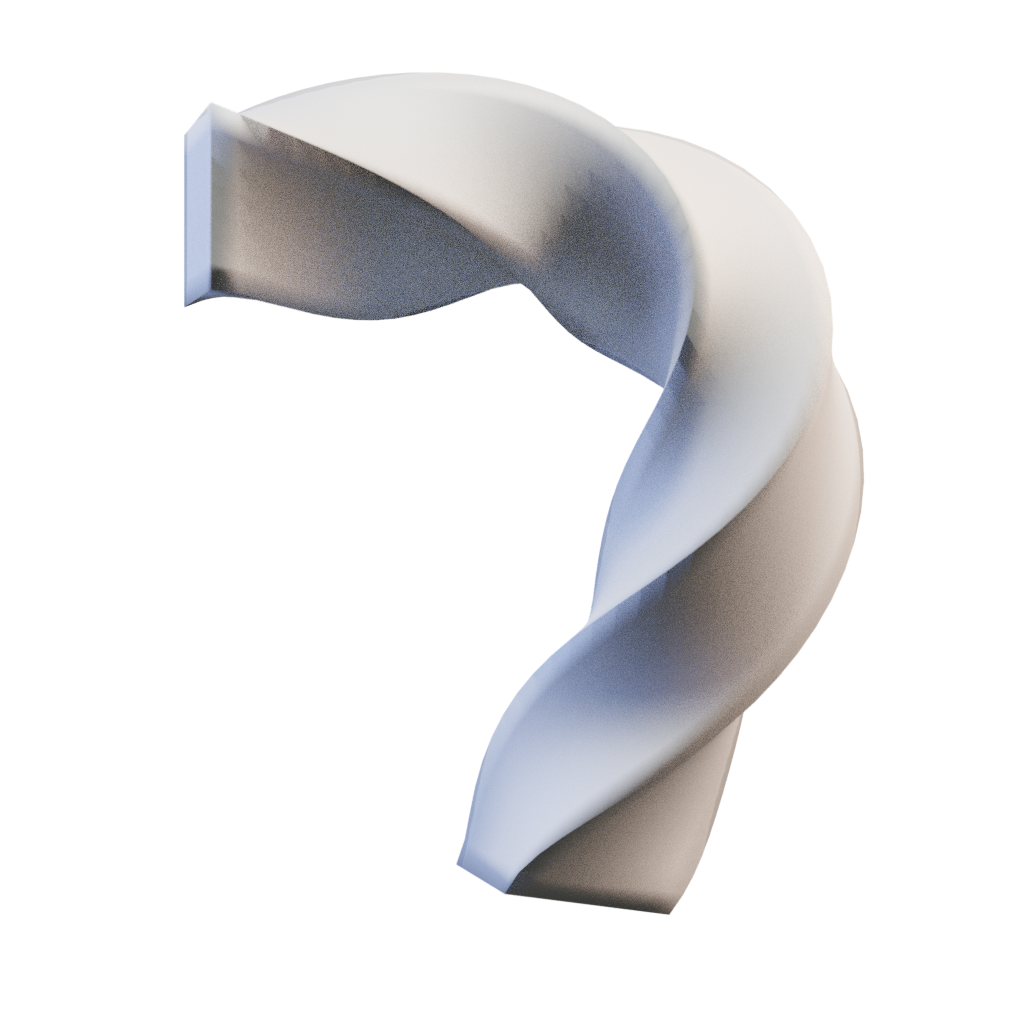}
    \hfill
    \includegraphics[width=.3\columnwidth]{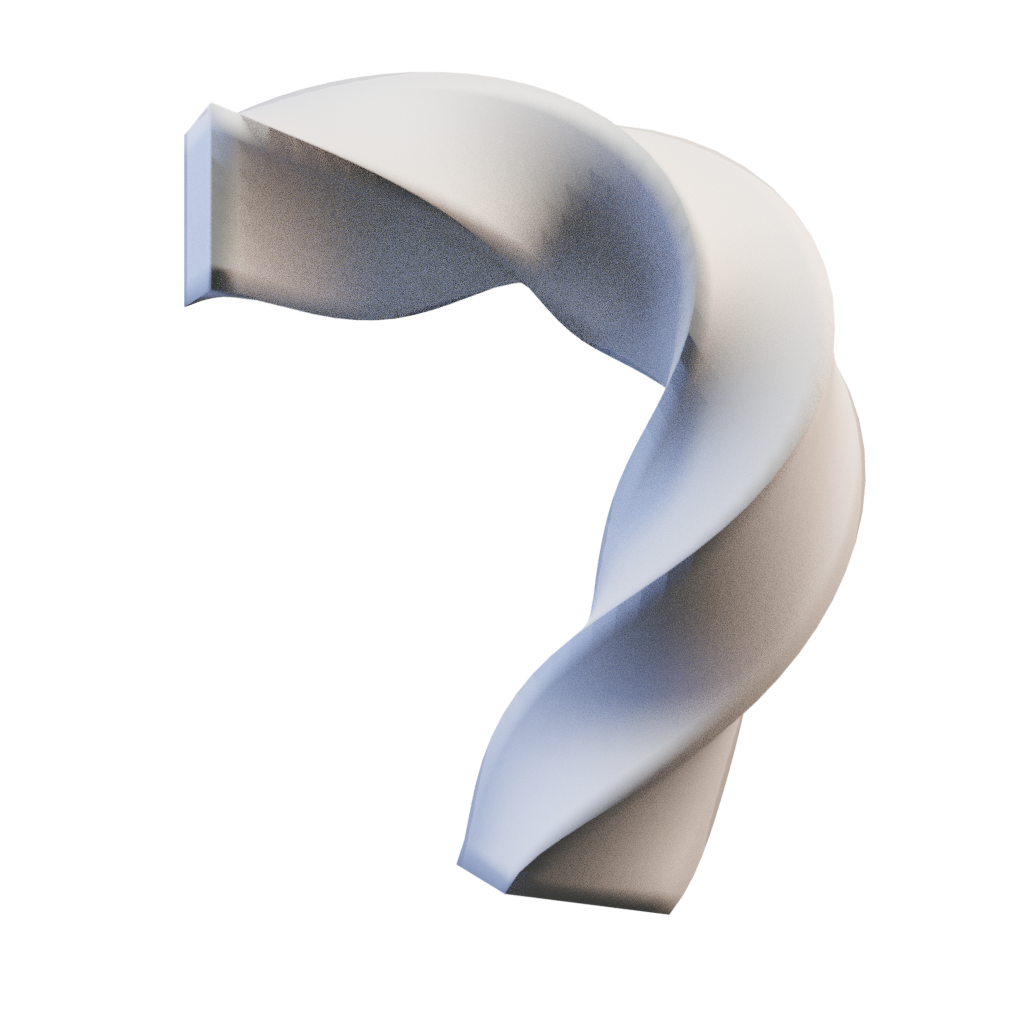}
    \hfill
    \includegraphics[width=.3\columnwidth]{img/cube-twist-arap.png}
    \caption{Twisting a horizontal bar by 360 degrees and then bending it, with
        three constitutive models from left to right: compressible neo-Hookean,
        incompressible neo-Hookean, and ARAP.
        \label{fig:cube-twist}
    }
\end{figure}

% f}}}

\subsection{Discussion}
% f{{{
\label{sec:meth-discuss}

Our experiments show that SANM delivers shorter solving times than Newton's
methods.  On the \speedupNrCases~comparison experiments (including both forward
equilibrium problems and controlled deformation problems), SANM achieves an
average speedup of \gmeanSpeedup~by geometric mean compared to the fastest
alternative method for each case. On average, SANM spends
\sparseSolverTimeUsed~of its running time in the sparse linear solver, while
most of the other time is used by Taylor coefficient computation that can be
further improved.

SANM and Newton's methods target different problems, and they can not completely
replace each other. SANM solves nonlinear systems via numerical continuation,
while Newton's methods typically solve minimization problems. Moreover,
numerical continuation allows easily computing intermediate equilibrium states
almost for free. By comparison, the intermediate states of a Newtonian solver
are less interpretable, but such solvers might reduce the energy in early
iterations and thus quickly produce visually plausible results.

Energy minimization does not apply to all of our experiment problems. We are
unaware of any global energy suitable for the inverse static equilibrium
problem. For controlled deformation problems, finding a proper initial guess for
energy minimization becomes nontrivial for certain energies that can not handle
inverted elements. We also need to take special care when applying energy
minimization to target configurations that involve ambiguity, such as rotations.
By contrast, SANM directly and efficiently handles these cases with numerical
continuation.

% f}}}

% vim: tw=80 filetype=tex foldmethod=marker foldmarker=f{{{,f}}} spell spelllang=en

\section{Conclusion}

The asymptotic numerical method is a powerful numerical continuation method for
solving nonlinear systems. Prior to our work, a major obstacle of applying ANM
was the difficulty in deriving the Taylor coefficients. We have shown that this
process can be fully automated and generalized to handle a large family of
nonlinearities. We also extend the ANM formulation to handle implicit homotopy.
Moreover, we implement an efficient and automatic ANM solver, SANM, that
delivers comparable or better performance than a hand-coded, manually optimized,
and specialized ANM solver. Although energy minimization targets different
problems from SANM in general, we also compare SANM with energy minimization via
Newton's methods on a few problems and show that SANM performs favorably.

With our tool, one can explore ANM on many applications in various fields with
little effort. SANM can contribute to improvements in numerical solving in many
systems. It may also inspire deeper theoretical understanding and further
improvement of ANM.

\begin{acks}
    We thank all the anonymous reviewers for providing the detailed review
    feedback, which has greatly improved the quality of this work. We would also
    like to thank Changxi Zheng, Wojciech Matusik, Liang Shi, and Martin Rinard
    for the helpful technical discussions and Lingxiao Li for proofreading.
    This work was funded by the project ``Automatically Learning the Behavior of
    Computational Agents'' (MIT CO 6940111, sponsored by Boeing with sponsor ID
    \#Z0918-5060).
\end{acks}

% vim: tw=80 filetype=tex foldmethod=marker foldmarker=f{{{,f}}} spell spelllang=en

% vim: filetype=tex foldmethod=marker foldmarker=f{{{,f}}} spell spelllang=en

%\input{repimg}

% Bibliography
\bibliographystyle{ACM-Reference-Format}
\bibliography{references}

\end{document}